\documentclass[12p]{amsart}
\usepackage{amssymb}
\usepackage{amsmath}
\usepackage{amsfonts}
\usepackage{geometry}
\usepackage{graphicx}
\usepackage{mathrsfs,amssymb}

\usepackage{hyperref}
\usepackage{cleveref}

\theoremstyle{plain}

\newtheorem{claim}{Claim}

\newtheorem{definition}{Definition}

\newtheorem{lemma}{Lemma}

\newtheorem{proposition}{Proposition}
\newtheorem{remark}{Remark}

\newtheorem{theorem}{Theorem}
\numberwithin{equation}{section}

\begin{document}
\title[]{Rigidity for the non self-dual Chern--Simons--Schr{\"o}dinger equation at the level of the soliton}

\author{Benjamin Dodson}
\date{\today}

\begin{abstract}
In this paper we prove a rigidity result for a solution to the non self-dual Chern--Simons--Schr{\"o}dinger equation at the level of the soliton.
\end{abstract}
\maketitle

\section{Introduction}
In this paper we prove rigidity for blowup solutions for the equivariant Chern--Simons--Schr{\"o}dinger equation,
\begin{equation}\label{1.1}
i u_{t} + \Delta u = \frac{2m}{r^{2}} A_{\theta}[u] u + A_{0}[u] u + \frac{1}{r^{2}} A_{\theta}[u]^{2} u - g |u|^{2} u, \qquad u : \mathbb{R} \times \mathbb{R}^{2} \rightarrow \mathbb{C}.
\end{equation}
where
\begin{equation}\label{1.2}
\aligned
A_{\theta}[u](t, r) &= -\frac{1}{2} \int_{0}^{r} |u(t, s)|^{2} s ds, \\
A_{0}[u](t, r) &= -\int_{r}^{\infty} (m + A_{\theta}[u](t, s)) |u(t, s)|^{2} \frac{ds}{s}.
\endaligned
\end{equation}
Here, we are in the equivariant case, which imposes the equivariant symmetry on the scalar field $\phi$,
\begin{equation}\label{1.3}
\phi(t, x) = u(t, r) e^{im \theta}, \qquad m \in \mathbb{Z}.
\end{equation}

This problem shares a number of similarities with the nonlinear Schr{\"o}dinger equation,
\begin{equation}\label{1.4}
i u_{t} + \Delta u = \alpha |u|^{2} u.
\end{equation}
Indeed, solutions to $(\ref{1.1})$ and $(\ref{1.2})$ conserve the quantities mass,
\begin{equation}\label{1.5}
M(u(t)) = \int |u(t, x)|^{2} dx = M(u(0)),
\end{equation}
and energy
\begin{equation}\label{1.6}
E(u(t)) = \frac{1}{2} \int |\partial_{r} u|^{2} + \frac{1}{2} \int (\frac{m + A_{\theta}[u]}{r})^{2} |u|^{2} - \frac{g}{4} \int |u|^{4}.
\end{equation}
Furthermore, $(\ref{1.1})$ is a mass-critical problem, as is $(\ref{1.4})$. Indeed, equation $(\ref{1.1})$ has the scaling symmetry
\begin{equation}\label{1.7}
u(t, x) \mapsto \lambda u(\lambda^{2} t, \lambda x), \qquad u_{0}(x) \mapsto \lambda u_{0}(\lambda x), \qquad \lambda > 0.
\end{equation}
The general Chern--Simons--Schr{\"o}dinger equation is locally well-posed \cite{liu2014local} for small data in $H_{x}^{s}(\mathbb{R}^{2})$, $s > 0$. The case when $s = 0$ is unknown. For the equivariant Chern--Simons--Schr{\"o}dinger equation, let $L_{m}^{2}$ denote the space of functions in $L^{2}(\mathbb{R}^{2})$ that satisfy $(\ref{1.3})$. Similarly, we can define the equivariant Sobolev spaces.
\begin{definition}[Equivariant Sobolev spaces]\label{d1.1}
Let $m \in \mathbb{Z}$. For each $s \geq 0$, define the function space $H_{m}^{s}$ to be the Sobolev space of functions $f \in H_{x}^{s}$ that admit the decomposition $f(x) = f(r, \theta) = e^{im \theta} u(r)$.
\end{definition}
Crucially, like the mass-critical nonlinear Schr{\"o}dinger equation, $(\ref{1.1})$ enjoys the virial identity
\begin{equation}\label{1.7.1}
\frac{d^{2}}{dt^{2}} \frac{1}{4} \int |x|^{2} |u|^{2} = \frac{d}{dt} \int Im[\bar{u} r \partial_{r} u] = 4 E[u],
\end{equation}
and the pseudoconformal transformation symmetry. If $u$ solves $(\ref{1.1})$ then
\begin{equation}\label{1.8.1}
\frac{1}{t} \overline{u(t, \frac{x}{t}) e^{i \frac{|x|^{2}}{4t}}},
\end{equation}
also solves $(\ref{1.1})$. Of course, by standard time translation and time reversal symmetry arguments, it is possible to replace the $t$ in $(\ref{1.8.1})$ by $T - t$. In this case, the pseudoconformal transformation may be abbreviated $PC_{T}$.\medskip

Rewriting the energy,
\begin{equation}\label{1.8}
E[u(t)] = \frac{1}{2} \int |\partial_{r} u - \frac{m + A_{\theta}[u]}{r} u|^{2} + \frac{1 - g}{4} \int |u|^{4}.
\end{equation}
Thus, when $g < 1$, $(\ref{1.1})$ resembles a defocusing nonlinear Schr{\"o}dinger equation ($\alpha > 0$), when $g > 1$, $(\ref{1.1})$ resembles a focusing nonlinear Schr{\"o}dinger equation ($\alpha < 0$), and when $g = 1$, $(\ref{1.1})$ is called a self-dual Chern--Simons--Schr{\"o}dinger equation.

\begin{theorem}\label{t1.3}
Let $g < 1$ and $m \in \mathbb{Z}$. Then $(\ref{1.1})$ is globally well-posed in $L_{m}^{2}$ and solutions scatter both forward and backward in time.
\end{theorem}
\begin{proof}
See \cite{liu2016global}. Compare to \cite{dodson2019defocusing} for the mass-critical NLS.
\end{proof}

\begin{theorem}\label{t1.4}
Let $g > 1$ and $m \in \mathbb{Z}_{+}$. Then there exists a constant $c_{m, g} > 0$ such that if $u_{0} \in L_{m}^{2}$ with $\| u_{0} \|_{L^{2}}^{2} < c_{m, g}$, then $(\ref{1.1})$ is globally well-posed in $L_{m}^{2}$ and scatters forward and backward in time. Moreover, the minimum charge of a nontrivial standing wave solution in the class $L_{t}^{\infty} L_{m}^{2}$ is equal to $c_{m, g}$.
\end{theorem}
\begin{proof}
See \cite{liu2016global}. Compare to \cite{dodson2015global} for the focusing, mass-critical NLS.
\end{proof}

\begin{theorem}[Self-dual case]\label{t1.5}
Let $g = 1$ and $m \in \mathbb{Z}_{+}$. Let $u_{0} \in L_{m}^{2}$ with $\| u_{0} \|_{L^{2}}^{2} < 8 \pi (m + 1)$. Then $(\ref{1.1})$ is globally well-posed in $L_{m}^{2}$ and scatters both forward and backward in time.
\end{theorem}
\begin{proof}
See \cite{liu2016global}.
\end{proof}

In this paper we prove a rigidity result that is analogous to the rigidity result for the mass-critical nonlinear Schr{\"o}dinger equation, see \cite{dodson2021determination} and \cite{dodson2021determination2}.
\begin{theorem}[Rigidity of blowup in finite time for $g > 1$]\label{t1.1}
For $m \in \mathbb{N}$ and $g > 1$, if $\| u_{0} \|_{L^{2}}^{2} = c_{m, g}$ then
\begin{equation}\label{1.9}
u = \psi^{(m, \alpha)}(t, x),
\end{equation}
 or $u = e^{i \gamma} PC_{T}[\lambda \psi^{(m, g)}(\lambda^{2} \cdot, \lambda \cdot)](t, x)$ for some $T > 0$.
\end{theorem}

For a generic $g > 1$, $(\ref{1.1})$ has a standing wave solution of the form
\begin{equation}\label{1.10}
\Delta u - \frac{2m}{r^{2}} A_{\theta}[u] u - A_{0}[u] u - \frac{1}{r^{2}} A_{\theta}[u]^{2} u + g |u|^{2} u -\alpha u = 0,
\end{equation}
for some $\alpha > 0$. This gives a standing wave equation to $(\ref{1.1})$ that is in the form $e^{i \alpha t} Q$, where $Q$ is the positive solution to $(\ref{1.10})$. The existence of such a solution was proved by \cite{2012standing}. The unique positive, standing wave solution to $(\ref{1.1})$ with a given $m$ and $\alpha > 0$ may be abbreviated $\psi^{(m, \alpha)}(t, x)$.\medskip

Theorem $\ref{t1.1}$ improves an earlier result of \cite{li2022threshold}.
\begin{theorem}\label{t1.2}
For $m \in \mathbb{N}$ and $g > 1$, if $\phi \in H_{m}^{1}(\mathbb{R}^{2})$, $\| u_{0} \|_{L^{2}}^{2} = c_{m, g}$, and the solution blows up forward in finite time at $T > 0$, there exists $\gamma \in [0, 2 \pi)$, $\lambda \in \mathbb{R}_{+}$, and an $m$-equivariant standing wave solution
\begin{equation}\label{1.11}
\psi^{(m, g)}(t, x) = e^{i \alpha t} \phi^{(m, g)}(x),
\end{equation}
such that
\begin{equation}\label{1.12}
u(t,x) = e^{i \gamma} PC_{T}[\lambda \psi^{(m, g)}(\lambda^{2} \cdot, \lambda \cdot)](t,x).
\end{equation}
\end{theorem}

The proof of Theorem $\ref{t1.1}$ is very similar to the argument in \cite{dodson2021determination2}. The main new difficulty is that equation $(\ref{1.1})$ is no longer a local equation.\medskip

The method proving Theorem $\ref{t1.1}$ does not extend to the self-dual, or $g = 1$ case. There are several reasons for this. The first is that the standing wave equation with $g = 1$ has $\alpha = 0$ in $(\ref{1.10})$. In the non self--dual case, the fact that $\alpha > 0$ is used extensively in the proof. Indeed, for the mass--critical problem, $(\ref{1.4})$, the soliton satisfies the elliptic equation
\begin{equation}
\Delta Q + |Q|^{2} Q = Q.
\end{equation}
Taking $u$ close to the soliton, $u = Q + \epsilon$,
\begin{equation}\label{1.12.1}
(\nabla Q, \nabla \epsilon) - (|Q|^{2} Q, \epsilon) = -(\Delta Q + |Q|^{2} Q, \epsilon) = -(Q, \epsilon) = \frac{1}{2} \| \epsilon \|_{L^{2}}^{2}.
\end{equation}
The last equality arises from the fact that $\| Q + \epsilon \|_{L^{2}} = \| Q \|_{L^{2}}$. Since $(\ref{1.12.1})$ represents the $\epsilon^{1}$ term in the expansion of $E[Q + \epsilon]$, we obtain
\begin{equation}\label{1.12.2}
E[Q + \epsilon] = \frac{1}{2}(\mathcal L \epsilon, \epsilon) + O(\epsilon^{3}),
\end{equation}
where $(\mathcal L \epsilon, \epsilon) \gtrsim \| \epsilon \|_{H^{1}}^{2}$ when $\epsilon$ is orthogonal to $\nabla Q$ and the negative eigenvector of $\mathcal L$. Since $Q$ is smooth and rapidly decreasing, $(\ref{1.12.1})$ is well--behaved under truncations in space and frequency. It is possible to obtain a similar estimate for $(\ref{1.10})$ when $\alpha > 0$. However, when $\alpha = 0$ we lose the $\| \epsilon \|_{L^{2}}^{2}$ term, which adds additional technical difficulties for a mass--critical problem.\medskip

Secondly, the standing wave solution to $(\ref{1.1})$ is no longer rapidly decreasing. This is also a by-product of the fact that $\alpha = 0$. Instead, the solution has the explicit form
\begin{equation}\label{1.13}
Q(r) = \sqrt{8} (m + 1) \frac{r^{m}}{1 + r^{2m + 2}}, \qquad m \geq 0.
\end{equation}
This fact is also used heavily. It seems likely to the author that Theorem $\ref{t1.1}$ should be true in the self-dual case, since \cite{li2022threshold} proved that Theorem $\ref{t1.1}$ does hold under the additional assumption that $u_{0} \in H_{m}^{1}$.\medskip

Additionally, it should be noted that when $m < 0$, \cite{kim2022soliton} proved that global well-posedness and scattering hold for initial data in $H_{m}^{1,1}$, and when $m \geq 0$, \cite{kim2022soliton} proved that a blowup solution should resolve into a single soliton plus a radiative term.
\begin{remark}
In fact, using the arguments proving Theorem $\ref{t1.1}$ in \cite{liu2016global}, when $m < 0$, a minimal mass blowup solution to $(\ref{1.1})$ can be reduced to one of three enemies:

\begin{itemize}
\item $N(t) = 1$,

\item $N(t) \leq 1$, $t \in \mathbb{R}$, $\liminf_{t \rightarrow \pm \infty} N(t) = 0$,

\item $N(t) = t^{-1/2}$, $t \in (0, \infty)$.
\end{itemize}

Also following the arguments in \cite{liu2016global}, it is possible to show that for an almost periodic solution to $(\ref{1.1})$, $u(t) \in H^{s}$ for $s \leq 2$, which furthermore implies that if $u$ is one of the three enemies, $E(u(t)) = 0$. However, using the estimate in \cite{kim2022soliton}, which shows that $E(u(t)) \sim_{M[u]} \| u \|_{\dot{H}_{m}^{1}}^{2}$ when $m < 0$, gives a contradiction.
\end{remark}

\section{Sequential convergence}
We begin with a sequential convergence result, comparable to the sequential convergence result for the mass-critical NLS in \cite{fan20182}, \cite{dodson20212}, and \cite{dodson2022l2}. The argument here follows the argument in \cite{dodson2023sequential} for the self-dual Chern--Simons--Schr{\"o}dinger equation.
\begin{theorem}[Sequential convergence]\label{t2.1}
Let $u$ be a solution to $(\ref{1.1})$ that blows up forward in time and satisfies $\| u \|_{L^{2}}^{2} = c_{m, g}$. That is,
\begin{equation}\label{2.1}
\lim_{T \nearrow \sup(I)} \| u \|_{L_{t,x}^{4}([0, T] \times \mathbb{R}^{2})} = +\infty.
\end{equation}
Then there exists $t_{n} \nearrow \sup(I)$ and sequences $\lambda(t_{n}) > 0$, $\gamma(t_{n}) \in [0, 2 \pi)$, such that
\begin{equation}\label{2.2}
e^{i \gamma(t_{n})} \lambda(t_{n}) u(t_{n}, \lambda(t_{n}) \cdot) \rightarrow \psi^{(m, g)}(\cdot), \qquad \text{in} \qquad L^{2},
\end{equation}
where $\psi^{(m, g)}$ is the real, positive standing wave solution to $(\ref{1.10})$.
\end{theorem}

The proof uses the fact that if $\| u \|_{L^{2}}^{2} = c_{m, g}$, then $u$ is a minimal mass blowup solution. Thus, it is possible to make use of much of the analysis in \cite{liu2016global}.

\begin{proposition}[Linear profile decomposition]\label{p2.2}
Let $\psi_{n}$, $n = 1, 2, ...$ be a bounded sequence in $L_{m}^{2}$. Then, after passing to a subsequence if necessary, there exists a sequence of functions $\phi^{j} \in L_{m}^{2}$, group elements $g_{n}^{j}$, and times $t_{n}^{j} \in \mathbb{R}$ such that we have the decomposition
\begin{equation}\label{2.3}
\psi_{n} = \sum_{j = 1}^{J} g_{n}^{j} e^{i t_{n}^{j} \Delta} \phi^{j} + w_{n}^{J}, \qquad \forall J = 1, 2, ...
\end{equation}
where $g_{n}^{j}$ belongs to the group of transformations of $L^{2}(\mathbb{R}^{2})$ generated by the scaling symmetry $(\ref{1.7})$ and multiplying by $e^{i \gamma}$ for some $\gamma \in \mathbb{R}$,
\begin{equation}\label{2.3.1}
g_{n}^{j} u(x) = e^{i \gamma(t_{n}^{j})} \lambda(t_{n}^{j}) u(\lambda(t_{n}^{j}) x).
\end{equation}

Moreover, $w_{n}^{J} \in L_{m}^{2}$ is such that its linear evolution has asymptotically vanishing scattering size
\begin{equation}\label{2.3.2}
\lim_{J \rightarrow \infty} \limsup_{n \rightarrow \infty} \| e^{it \Delta} w_{n}^{J} \|_{L_{t,x}^{4}} = 0.
\end{equation}
Moreover, for any $j \neq j'$,
\begin{equation}\label{2.3.3}
\frac{\lambda_{n}^{j}}{\lambda_{n}^{j'}} + \frac{\lambda_{n}^{j'}}{\lambda_{n}^{j}} + \frac{|t_{n}^{j} (\lambda_{n}^{j})^{2} - t_{n}^{j'} (\lambda_{n}^{j'})^{2}|}{\lambda_{n}^{j} \lambda_{n}^{j'}} \rightarrow \infty.
\end{equation}
Furthermore, for any $J \geq 1$, we have the mass decoupling property
\begin{equation}\label{2.4}
\lim_{n \rightarrow \infty} [M(u(t_{n})) - \sum_{j = 1}^{J} M(\phi^{j}) - M(w_{n}^{J})] = 0,
\end{equation}
\end{proposition}
\begin{proof}
This is Proposition $3.1$ of \cite{liu2016global}.
\end{proof}

Specifically, let $t_{n} \nearrow \sup(I)$ be a sequence and let $\psi_{n} = u(t_{n})$ and apply Proposition $\ref{p2.2}$. Then possibly after passing to a subsequence,

\begin{claim}
If $u$ is a blowup solution to $(\ref{1.1})$, there exists some $j$ such that $\phi^{j} \neq 0$.
\end{claim}
\begin{proof}
Otherwise, by a perturbative argument, $(\ref{2.3.2})$ implies that $u$ is a scattering solution. Relabeling, suppose $\phi^{1} \neq 0$.
\end{proof}

\begin{claim}
$\| \phi^{j} \|_{L^{2}} = 0$ for $j > 1$. 
\end{claim}
\begin{proof}
Otherwise, by $(\ref{2.4})$, if $\| \phi^{2} \|_{L^{2}} > 0$ for $j > 2$, then $\| \phi^{j} \|_{L^{2}}^{2} < c_{m, g}$ for all $j$. By \cite{liu2016global} and standard perturbative arguments, $u$ scatters forward in time.
\end{proof}

\begin{claim}
For any $J > 1$,
\begin{equation}\label{2.5}
\lim_{n \rightarrow \infty} \| w_{n}^{J} \|_{L^{2}} = 0.
\end{equation}
\end{claim}
\begin{proof}
Otherwise by $(\ref{2.4})$, if $\limsup_{n \rightarrow \infty} \| w_{n}^{J} \|_{L^{2}} > 0$, and therefore $\| \phi^{1} \|_{L^{2}}^{2} < c_{m, g}$. Then by standard perturbative arguments, $u$ scatters forward in time.
\end{proof}

\begin{claim}
After possibly passing to a subsequence, the sequence $t_{n}^{1}$ converges as $n \rightarrow \infty$. 
\end{claim}
\begin{proof}
If $t_{n}^{1} \rightarrow +\infty$ then we have scattering forward in time. If $t_{n}^{1} \rightarrow -\infty$, we have scattering backward in time, which contradicts
\begin{equation}\label{2.6}
\| u \|_{L_{t,x}^{4}(\inf(I), t_{n}] \times \mathbb{R}^{2})} \rightarrow +\infty,
\end{equation}
as $n \rightarrow \infty$.
\end{proof}

Therefore, possibly after passing to a subsequence,
\begin{equation}\label{2.7}
(g_{n}^{1})^{-1} u(t_{n}, x) \rightarrow \phi^{1}, \qquad \phi^{1} \in L^{2}, \qquad \| \phi^{1} \|_{L^{2}} = c_{m, g}.
\end{equation}
Now then, by construction, $\phi^{1}$ is the initial value of a blowup solution to $(\ref{1.1})$ that blows up both forward and backward in time. Let $\phi$ be the solution to $(\ref{1.1})$ with initial data $\phi^{1}$. Since $\phi$ is a minimal mass blowup solution, then after making the concentration compactness argument in \cite{liu2016global}, there exists $\lambda(t)$, $\gamma(t)$ such that
\begin{equation}\label{2.8}
e^{i \gamma(t)} \lambda(t) \phi(t, \lambda(t) \cdot) \in K \subset L^{2},
\end{equation}
where $K \subset L^{2}$ is a precompact set. Furthermore, following the reduction to three enemies in  \cite{liu2016global}, see also \cite{killip2009cubic}. there exist $t_{n}$ such that
\begin{equation}\label{2.9}
\lambda(t_{n}) e^{i \gamma(t_{n})} \phi(t_{n}, \lambda(t_{n}) \cdot) \rightarrow v_{0} \in L^{2}.
\end{equation}
Furthermore, $v_{0}$ is the initial data for a solution to $(\ref{1.1})$ that satisfies $(\ref{2.8})$ and $\lambda(t)$ satisfies one of three cases:

\begin{itemize}
\item $\lambda(t) = 1$ for all $t \in \mathbb{R}$,

\item $\lambda(t) \geq 1$ for all $t \in \mathbb{R}$ and $\limsup_{t \rightarrow \pm \infty} \lambda(t) = \infty$,

\item $\lambda(t) = t^{1/2}$ for $t \in (0, \infty)$.
\end{itemize}

Using the additional regularity argument in \cite{liu2016global}, $E(v) = 0$, so since $\| v \|_{L^{2}}^{2} = c_{m, g}$, $v_{0}$ is a soliton. See Proposition $3.7$ of \cite{li2022threshold} for the proof that the solitons are the only zero energy $m$-equivariant functions satisfying $\| v \|_{L^{2}}^{2} = c_{m, g}$. Therefore,
\begin{equation}\label{2.10}
e^{i \gamma(\tau_{n'})} \lambda(\tau_{n'}) \phi(\tau_{n'}, \lambda(\tau_{n'}) \cdot) \rightarrow \psi^{(m, g)}, \qquad \text{in} \qquad L^{2}.
\end{equation}
Therefore, choosing $n(n')$ sufficiently large,
\begin{equation}\label{2.11}
e^{i \gamma(\tau_{n'})} \lambda(\tau_{n'}) (g_{n}^{1})^{-1} u(t_{n} + \lambda(t_{n})^{-2} \tau_{n'}, \lambda(\tau_{n'}) \cdot) \rightarrow \psi^{(m, g)}, \qquad \text{in} \qquad L^{2}.
\end{equation}

\section{Reduction of a blowup solution}

Fix some $0 < \eta_{\ast} \ll 1$ sufficiently small. To prove Theorem $\ref{t1.2}$, it suffices to prove the following. To simplify notation let $Q = \psi^{(m, g)}$ be the positive solution to the standing wave equation $(\ref{1.10})$. Theorem $\ref{t1.1}$ can be reduced to Theorem $\ref{t3.1}$.

\begin{theorem}\label{t3.1}
If $u$ is a blowup solution to $(\ref{1.1})$ that satisfies $\| u \|_{L^{2}}^{2} = c_{m, g}$ and for all $t \geq 0$,
\begin{equation}\label{3.1}
\inf_{\lambda > 0, \gamma \in \mathbb{R}} \| e^{i \gamma} \lambda u(t, \lambda x) - Q \|_{L^{2}} \leq \eta_{\ast},
\end{equation}
then Theorem $\ref{t1.1}$ is true.
\end{theorem}

\begin{proof}[Theorem $\ref{t3.1}$ implies Theorem $\ref{t1.1}$]
Suppose that $u$ is a solution to $(\ref{1.1})$ that blows up forward in time and satisfies $\| u \|_{L^{2}}^{2} = c_{m, g}$. Consider two cases separately.\medskip

\noindent \textbf{Case 1:} There exists some $t_{0} > 0$ such that, for all $t \in [t_{0}, \sup(I))$,
\begin{equation}\label{3.2}
\inf_{\lambda > 0, \gamma \in \mathbb{R}} \| e^{i \gamma} \lambda u(t, \lambda \cdot) - Q \|_{L^{2}} \leq \eta_{\ast}.
\end{equation}
In this case, Theorem $\ref{t3.1}$ reduces to Theorem $\ref{t1.1}$.\medskip

\noindent \textbf{Case 2:} There exist a sequences $t_{n}, t_{n}^{-} \nearrow \sup(I)$ such that
\begin{equation}\label{3.3}
\sup_{t \in [t_{n}^{-}, t_{n}]} \inf_{\lambda > 0, \gamma \in \mathbb{R}} \| e^{i \gamma} \lambda u(t_{n}, \lambda \cdot) - Q \|_{L^{2}} \leq \eta_{\ast},
\end{equation}
\begin{equation}\label{3.4}
\inf_{\lambda > 0, \gamma \in \mathbb{R}} \| e^{i \gamma} \lambda u(t_{n}^{-}, \lambda \cdot) - Q \|_{L^{2}} = \eta_{\ast},
\end{equation}
and
\begin{equation}\label{3.5}
\lim_{n \rightarrow \infty} \| u \|_{L_{t,x}^{4}((\inf(I), t_{n}^{-}] \times \mathbb{R}^{2})} = \lim_{n \rightarrow \infty} \| u \|_{L_{t,x}^{4}([t_{n}^{-}, t_{n}] \times \mathbb{R}^{2})} = \infty.
\end{equation}

To see why $(\ref{3.3})$--$(\ref{3.5})$ must hold for a blowup solution to $(\ref{1.1})$ that does not satisfy Case 1, observe that by Theorem $\ref{t2.1}$, there exists a sequence $t_{n} \nearrow \sup(I)$ such that $(\ref{2.2})$ holds. Since $(\ref{3.2})$ does not hold, we also have $t_{n}^{-} \nearrow \sup(I)$.\medskip

Next, recall the Strichartz estimates of \cite{yajima1987existence}, \cite{ginibre1992smoothing}, and \cite{tao2000spherically}.

\begin{lemma}[Strichartz estimates]\label{l3.2}
Let $(i \partial_{t} + \Delta) u = f$ on a time interval $I$ with $t_{0} \in I$ and $u(t_{0}) = u_{0}$. A pair $(p, q)$ of exponents is called admissible if $2 \leq p, q \leq \infty$, $\frac{1}{p} + \frac{1}{q} = \frac{1}{2}$, and $(p, q) \neq (2, \infty)$. Let $(p, q)$ and $(\tilde{p}, \tilde{q})$ be admissible pairs of exponents. Then,
\begin{equation}\label{3.6}
\| u \|_{L_{t}^{\infty} L_{x}^{2}(I \times \mathbb{R}^{2})} + \| u \|_{L_{t}^{p} L_{x}^{q}(I \times \mathbb{R}^{2})} \lesssim \| u_{0} \|_{L^{2}} + \| f \|_{L_{t}^{\tilde{p}'} L_{x}^{\tilde{q}'}(I \times \mathbb{R}^{2})}.
\end{equation}
\end{lemma}

\begin{lemma}[Endpoint Strichartz estimates]\label{l3.3}
Let $(i \partial_{t} + \Delta) u = f$ on a time interval $I$ with $t_{0} \in I$ and $u(t_{0}) = u_{0}$, and suppose that $m \in \mathbb{Z}$ and $u_{0} \in L_{m}^{2}$, $f \in L_{t}^{1} L_{m}^{2}(I \times \mathbb{R}^{2})$. Let $(p, q)$ be an admissible pair of exponents. Then,
\begin{equation}\label{3.7}
\| u \|_{L_{t}^{2} L_{x}^{\infty}(I \times \mathbb{R}^{2})} \lesssim \| u_{0} \|_{L^{2}} + \| f \|_{L_{t}^{p'} L_{x}^{q'}(I \times \mathbb{R}^{2})}.
\end{equation}
\end{lemma}

Finally,
\begin{lemma}[Control of the nonlinearity]\label{l3.4}
Let
\begin{equation}\label{3.8}
\Lambda(u) =  \frac{2m}{r^{2}} A_{\theta}[u] u + A_{0}[u] u + \frac{1}{r^{2}} A_{\theta}[u]^{2} u - g |u|^{2} u.
\end{equation}
We have
\begin{equation}\label{3.9}
\| \Lambda(u) \|_{L_{t,x}^{4/3}(I \times \mathbb{R}^{2})} \lesssim \| u \|_{L_{t,x}^{4}(I \times \mathbb{R}^{2})}^{3},
 \end{equation}
 and
 \begin{equation}\label{3.10}
 \| \Lambda(u) - \Lambda(\tilde{u}) \|_{L_{t,x}^{4/3}(I \times \mathbb{R}^{2})} \lesssim \| u - \tilde{u} \|_{L_{t,x}^{4}(I \times \mathbb{R}^{2})} (\| u \|_{L_{t,x}^{4}(I \times \mathbb{R}^{2})}^{2} + \| \tilde{u} \|_{L_{t,x}^{4}(I \times \mathbb{R}^{2})}^{2}).
 \end{equation}
\end{lemma}
\begin{proof}
This is proved in \cite{liu2016global}.
\end{proof}

It follows from Lemma $\ref{l3.4}$ that
\begin{equation}\label{3.11}
\inf_{\lambda > 0, \gamma \in \mathbb{R}} \| e^{i \gamma} \lambda u(t, \lambda x) - Q \|_{L^{2}},
\end{equation}
is continuous as a function in $t$. Therefore, for each $t_{n} \in I$, there exists some $t_{n}^{-} \in I$, $t_{n}^{-} < t_{n}$, such that
\begin{equation}\label{3.12}
\inf_{\lambda > 0, \gamma \in \mathbb{R}} \| e^{i \gamma} \lambda u(t_{n}^{-}, \lambda x) - Q \|_{L^{2}} = \eta_{\ast},
\end{equation}
and
\begin{equation}\label{3.13}
\sup_{t \in [t_{n}^{-}, t_{n}]} \inf_{\lambda > 0, \gamma \in \mathbb{R}} \| e^{i \gamma} \lambda u(t_{n}^{-}, \lambda x) - Q \|_{L^{2}} = \eta_{\ast}.
\end{equation}
Thus, $(\ref{3.3})$ and $(\ref{3.4})$ hold. Finally, using the perturbation result in Lemma $\ref{l3.4}$, $(\ref{3.9})$,
\begin{equation}\label{3.14}
\inf_{\lambda > 0, \gamma \in \mathbb{R}} \| e^{i \gamma} \lambda u(t', \lambda x) - Q \|_{L^{2}} \lesssim \inf_{\lambda > 0, \gamma \in \mathbb{R}} \| e^{i \gamma} \lambda u(t, \lambda x) - Q \|_{L^{2}},
\end{equation}
with implicit constant depending only on $u$, for any pair of times $t, t'$ such that
\begin{equation}\label{3.15}
\| u \|_{L_{t,x}^{4}([t, t'] \times \mathbb{R}^{2})} \leq 1.
\end{equation}
Since
\begin{equation}\label{3.16}
\inf_{\lambda > 0, \gamma \in \mathbb{R}} \| e^{i \gamma} \lambda u(t_{n}, \lambda x) - Q \|_{L^{2}} \rightarrow 0,
\end{equation}
$(\ref{3.14})$ and $(\ref{3.15})$ imply $(\ref{3.5})$.\medskip

Now, using Proposition $\ref{p2.2}$, there exists a sequence $g_{n} \in G$ and $u_{0} \in L^{2}$, $\| u_{0} \|_{L^{2}}^{2} = c_{m, g}$, such that
\begin{equation}\label{3.17}
g_{n}^{-1} u(t_{n}^{-1}, x) \rightarrow u_{0}, \qquad \text{in} \qquad L^{2}.
\end{equation}
Furthermore, by $(\ref{3.3})$--$(\ref{3.5})$, $u_{0}$ is the initial data to a solution to $(\ref{1.1})$ that blows up both forward and backward in time, satisfies $(\ref{3.2})$, and satisfies
\begin{equation}\label{3.18}
\inf_{\lambda > 0, \gamma \in \mathbb{R}} \| e^{i \gamma} \lambda u_{0}(\lambda x) - Q \|_{L^{2}} = \eta_{\ast}.
\end{equation}
But then by Theorem $\ref{t3.1}$, $u$ must be a pseudoconformal transformation of a soliton. However, this gives a contradiction, since pseudoconformal transformations of a soliton blow up in one time direction and scatter in the other. Therefore, Case 2 cannot happen.
\end{proof}

\section{Decomposition of the energy}
Now decompose the energy. Recall that
\begin{equation}\label{4.1}
E[u] = \frac{1}{2} \int |(\partial_{r} - \frac{m + A_{\theta}[u]}{r}) u|^{2} dx + \frac{1 - g}{4} \int |u|^{4}.
\end{equation}
Now then, let $u = Q + \epsilon$ for $\| \epsilon \|_{L^{2}} \ll 1$, $\epsilon$ is real valued.
\begin{equation}\label{4.2}
\aligned
\| (\partial_{r} - \frac{m + A_{\theta}[u]}{r}) u \|_{L^{2}} = \| (\partial_{r} - \frac{m + A_{\theta}[Q]}{r}) Q + (\frac{Re \int_{0}^{r} Q \bar{\epsilon} s ds}{r}) Q + (\partial_{r} - \frac{m + A_{\theta}[Q]}{r}) \epsilon \|_{L^{2}} \\ + O(\| \epsilon \|_{L^{2}} \| \epsilon \|_{\dot{H}_{m}^{1}} + \| \epsilon \|_{L^{2}}^{2}).
\endaligned
\end{equation}
Indeed, decompose
\begin{equation}\label{4.3}
(\partial_{r} - \frac{m + A_{\theta}[u]}{r})(Q + \epsilon) = (\partial_{r} - \frac{m + A_{\theta}[Q]}{r})Q - (\frac{A_{\theta}[u] - A_{\theta}[Q]}{r}) Q + (\partial_{r} - \frac{m + A_{\theta}[Q]}{r}) \epsilon - (\frac{A_{\theta}[u] - A_{\theta}[Q]}{r}) \epsilon.
\end{equation}
By direct computation,
\begin{equation}\label{4.4}
\| (\frac{A_{\theta}[u] - A_{\theta}[Q]}{r}) \epsilon\|_{L^{2}} = \| \frac{1}{2} (\int_{0}^{r} [|u|^{2} - Q^{2}] s ds) \frac{\epsilon}{r} \|_{L^{2}} \lesssim \| \epsilon \|_{L^{2}} \| \epsilon \|_{\dot{H}_{m}^{1}}.
\end{equation}
Meanwhile,
\begin{equation}\label{4.5}
-(\frac{A_{\theta}[u] - A_{\theta}[Q]}{r}) Q = \frac{1}{r} Re (\int_{0}^{r} Q \bar{\epsilon} s ds)Q + \frac{1}{2r} (\int_{0}^{r} |\epsilon|^{2} s ds) Q.
\end{equation}
Again by direct computation,
\begin{equation}\label{4.6}
\| \frac{Q}{2r} (\int_{0}^{r} |\epsilon|^{2} s ds) \|_{L^{2}} \lesssim \| \epsilon \|_{L^{2}}^{2}.
\end{equation}
Expanding
\begin{equation}\label{4.7}
\aligned
\frac{1}{2} \| (\partial_{r} - \frac{m + A_{\theta}[Q]}{r}) Q + (\frac{Re \int_{0}^{r} Q \bar{\epsilon} s ds}{r}) Q + (\partial_{r} - \frac{m + A_{\theta}[Q]}{r}) \epsilon \|_{L^{2}}^{2} + \frac{1 - g}{4} \| u \|_{L^{4}}^{4} \\
= \frac{1}{2} \| (\partial_{r} - \frac{m + A_{\theta}[Q]}{r}) Q \|_{L^{2}}^{2} + \frac{1 - g}{4} \| Q \|_{L^{4}}^{4} + \langle (\partial_{r} - \frac{m + A_{\theta}[Q]}{r}) Q, (\frac{Re \int_{0}^{r} Q \bar{\epsilon} s ds}{r}) Q + (\partial_{r} - \frac{m + A_{\theta}[Q]}{r}) \epsilon \rangle \\ + (1 - g) Re \int Q^{3} \bar{\epsilon} + \frac{1}{2} \|  (\frac{Re \int_{0}^{r} Q \bar{\epsilon} s ds}{r}) Q + (\partial_{r} - \frac{m + A_{\theta}[Q]}{r}) \epsilon \|_{L^{2}}^{2} + 3(1 - g) \int Q^{2} \epsilon^{2} \\ + O(\| \epsilon \|_{L^{2}}^{5/2} \| \epsilon \|_{\dot{H}_{m}^{1}}^{1/2} + \| \epsilon \|_{\dot{H}_{m}^{1}}^{2} \| \epsilon \|_{L^{2}}^{2}).
\endaligned
\end{equation}
Since $E[Q] = 0$,
\begin{equation}\label{4.8}
\frac{1}{2} \| (\partial_{r} - \frac{m + A_{\theta}[Q]}{r}) Q \|_{L^{2}}^{2} + \frac{1 - g}{4} \| Q \|_{L^{4}}^{4} = 0.
\end{equation}
Next, integrating by parts,
\begin{equation}\label{4.9}
\aligned
\langle (\partial_{r} - \frac{m + A_{\theta}[Q]}{r}) Q, (\frac{Re \int_{0}^{r} Q \bar{\epsilon} s ds}{r}) Q + (\partial_{r} - \frac{m + A_{\theta}[Q]}{r}) \epsilon \rangle \\ = \langle -\partial_{rr} Q - \frac{1}{r} \partial_{r} Q + \partial_{r}(A_{\theta}[Q]) \frac{Q}{r} + (\frac{m + A_{\theta}[Q]}{r})^{2} Q - \frac{1}{2} Q^{3}, \epsilon \rangle - \int \frac{m + A_{\theta}[Q]}{r^{2}} Q^{2} Re(\int_{0}^{r} Q \bar{\epsilon} s ds) r dr \\
= \langle -\partial_{rr} Q - \frac{1}{r} \partial_{r} Q - Q^{3} + (\frac{m + A_{\theta}[Q]}{r})^{2} Q + A_{0}[Q] Q, \epsilon \rangle.
\endaligned
\end{equation}
Since
\begin{equation}\label{4.10}
(\partial_{r}^{2} + \frac{1}{r} \partial_{r} - \alpha - (\frac{m + A_{\theta}[Q]}{r})^{2} - A_{0}[Q] - g Q^{2})Q = 0,
\end{equation}
\begin{equation}\label{4.11}
\langle  -\partial_{rr} Q - \frac{1}{r} \partial_{r} Q - Q^{3} + (\frac{m + A_{\theta}[Q]}{r})^{2} Q + A_{0}[Q] Q, \epsilon \rangle + (1 - g) \langle Q^{3}, \epsilon \rangle = -\alpha \langle Q, \epsilon \rangle = \frac{\alpha}{2} \| \epsilon \|_{L^{2}}^{2}.
\end{equation}
The last equality uses the fact that $\| Q + \epsilon \|_{L^{2}} = \| Q \|_{L^{2}}$.\medskip

Now let
\begin{equation}\label{4.12}
L_{Q} f = (\frac{Re \int_{0}^{r} Q \bar{f} s ds}{r}) Q + (\partial_{r} - \frac{m + A_{\theta}[Q]}{r}) f.
\end{equation}
We have proved
\begin{equation}\label{4.13}
E[Q + \epsilon] = \frac{\alpha}{2} \| \epsilon \|_{L^{2}}^{2} + \frac{1}{2} \| L_{Q} \epsilon \|_{L^{2}}^{2} + \frac{3}{2}(1 - g) \int Q^{2} \epsilon^{2} + O(\| \epsilon \|_{\dot{H}_{m}^{1}}^{2} \| \epsilon \|_{L^{2}}^{2} + \| \epsilon \|_{L^{2}}^{4} + \| \epsilon \|_{\dot{H}_{m}^{1}}^{1/2} \| \epsilon \|_{L^{2}}^{5/2}).
\end{equation}
Therefore, if
\begin{equation}\label{4.14}
\| L_{Q} \epsilon \|_{L^{2}}^{2} + 3(1 - g) \int Q^{2} \epsilon^{2} \gtrsim \| \epsilon \|_{\dot{H}_{m}^{1}}^{2},
\end{equation}
then
\begin{equation}\label{4.15}
E[Q + \epsilon] \gtrsim \| \epsilon \|_{L^{2}}^{2} + \| \epsilon \|_{\dot{H}_{m}^{1}}^{2}.
\end{equation}

\begin{claim}
We claim that there exists a rapidly decreasing negative eigenfunction $\psi$ of $(\ref{4.14})$ that satisfies
\begin{equation}\label{4.16}
\langle Q + x \cdot \nabla Q, \psi \rangle \neq 0.
\end{equation}
Furthermore, if $\epsilon \perp \psi$, then $(\ref{4.14})$ holds.
\end{claim}

\begin{remark}
Using Lemma $2.1$ of \cite{kim2022soliton},
\begin{equation}
\| L_{Q} f \|_{L^{2}} \sim \| f \|_{\dot{\mathcal H}_{m}^{1}},
\end{equation}
for $f$ orthogonal to $\mathcal Z_{1}$ and $\mathcal Z_{2}$ satisfying the transversality condition. Then for $g$ sufficiently close to $1$, $(\ref{4.14})$ holds. The transversality condition from \cite{kim2022soliton} is that
\begin{equation}
det \begin{pmatrix}
(\Lambda Q, \mathcal Z_{1}) & (iQ, \mathcal Z_{1}) \\
(\Lambda Q, \mathcal Z_{2}) & (iQ, \mathcal Z_{2})
\end{pmatrix} \neq 0, \qquad \Lambda = 1 + x \cdot \nabla,
\end{equation}
so $(\ref{4.16})$ certainly holds. We conjecture that $(\ref{4.14})$ holds for any $g > 1$.
\end{remark}

If the above claim is true, then by the implicit function theorem we can prove the following.
\begin{theorem}\label{t4.1}
Suppose there exists some $\lambda_{0}(t) > 0$, $\gamma_{0}(t) \in \mathbb{R}$ such that
\begin{equation}\label{4.17}
\| e^{i \gamma_{0}(t)} \lambda_{0}(t) u(t, \lambda_{0}(t) x) - Q(x) \|_{L^{2}} \leq \eta_{0}.
\end{equation}
Suppose without loss of generality that $\lambda_{0}(t) = 1$ and $\gamma_{0}(t) = 0$. Then there exists $\gamma(t) \in \mathbb{R}$, $\lambda(t) > 0$ such that
\begin{equation}\label{4.17}
\| e^{i \gamma(t)} \lambda(t) u(t, \lambda(t) x) - Q(x) \|_{L^{2}} \leq 2\eta_{0},
\end{equation}
\begin{equation}\label{4.18}
\langle e^{i \gamma(t)} \lambda(t) u(t, \lambda(t) x) - Q(x), \psi \rangle = \langle e^{i \gamma(t)} \lambda(t) u(t, \lambda(t) x) - Q(x), i \psi \rangle = 0,
\end{equation}
and
\begin{equation}\label{4.19}
|\gamma(t)| + |\lambda(t) - 1| \lesssim \| \epsilon \|_{L^{2}} + \| \epsilon \|_{L^{2}} \| \epsilon \|_{L^{4}}^{2}.
\end{equation}

\end{theorem}

\section{Long time Strichartz estimates}
\subsection{Estimates for the mass-critical NLS}
As a warm-up, we prove an estimate for a mass-critical NLS.
\begin{proposition}\label{p5.1}
Suppose $u$ is a solution to the mass-critical nonlinear Schr{\"o}dinger equation,
\begin{equation}\label{5.1}
i u_{t} + \Delta u = -|u|^{2} u, \qquad u(0, x) = u_{0}, \qquad \| u_{0} \|_{L^{2}} = \| Q \|_{L^{2}}.
\end{equation}
Furthermore, suppose that for some interval $[a, b]$ with $a > 0$,
\begin{equation}\label{5.2}
\int_{a}^{b} \lambda(t)^{-2} dt = T, \qquad T > \eta_{\ast}^{-1},
\end{equation}
\begin{equation}\label{5.3}
\sup_{t > 0} dist(u(t), \mathcal M) \leq \eta_{\ast},
\end{equation}
and that $u$ is equivariant of order $m$. Finally, suppose that for all $t \in [a, b]$,
\begin{equation}\label{5.4}
1 \leq \lambda(t) \leq T^{1/100}.
\end{equation}
Then,
\begin{equation}\label{5.5}
\int_{a}^{b} \| \epsilon(t) \|_{L^{2}}^{2} \lambda(t)^{-2} dt \leq 3(\epsilon_{2}(a), Q + x \cdot \nabla Q)_{L^{2}} - (3 \epsilon_{2}(b), Q + x \cdot \nabla Q)_{L^{2}} + O(T^{-8}).
\end{equation} 
\end{proposition}
\begin{proof}
It suffices to prove the proposition under the bootstrap assumption,
\begin{equation}\label{5.6}
\int_{a}^{b} \| \epsilon(t) \|_{L^{2}}^{2} \lambda(t)^{-2} dt \leq \eta_{\ast}^{1/2}.
\end{equation}
Indeed, since $\| \epsilon(t) \|_{L^{2}}^{2} \lesssim \eta_{\ast}^{2}$ for all $t \in [a, b]$,
\begin{equation}\label{5.7}
\int_{a}^{b'} \| \epsilon(t) \|_{L^{2}}^{2} dt \leq \eta_{\ast}^{1/2},
\end{equation}
where
\begin{equation}\label{5.8}
\int_{a}^{b'} \lambda(t)^{-2} dt \sim \epsilon_{\ast}^{-3/2} \gg \epsilon_{\ast}^{-1}.
\end{equation}
If $b' \geq b$, then the proof is complete. If $b' < b$, then since $(\ref{5.2})$--$(\ref{5.4})$ hold on $[a, b']$, $(\ref{5.5})$ implies
\begin{equation}\label{5.9}
\int_{a}^{b'} \| \epsilon(t) \|_{L^{2}}^{2} \lambda(t)^{-2} dt \lesssim \eta_{\ast} \ll \eta_{\ast}^{1/2}.
\end{equation}
By standard bootstrap arguments, the proof of Proposition $\ref{p5.1}$ would then be complete.\medskip

The proof of Proposition $\ref{p5.1}$ depends on two propositions: the long time Strichartz estimate and  the almost conservation of energy.
\begin{proposition}[Long time Strichartz estimates]\label{p5.2}
Under the conditions of Proposition $\ref{p5.1}$ and $(\ref{5.6})$, if $N = T^{1/3}$,
\begin{equation}\label{5.10}
\| P_{> N} u \|_{U_{\Delta}^{2}([a, b] \times \mathbb{R}^{2})}^{2} \lesssim \frac{1}{T^{10}} + \frac{1}{T} \int_{a}^{b} \| \epsilon(t) \|_{L^{2}}^{2} \lambda(t)^{-2} dt.
\end{equation}
\end{proposition}

\begin{proposition}[Almost conservation of energy]\label{p5.3}
Under the conditions of Proposition $\ref{p5.1}$ and $(\ref{5.6})$, if $N = T^{1/3}$,
\begin{equation}\label{5.11}
\sup_{t \in [a, b]} E(P_{\leq N} u)(t) \lesssim \frac{1}{T^{10}} + \frac{N^{2}}{T} \int_{a}^{b} \| \epsilon(t) \|_{L^{2}}^{2} \lambda(t)^{-2} dt.
\end{equation}
\end{proposition}
Let $\chi \in C_{0}^{\infty}(\mathbb{R}^{2})$ be a radially symmetric function, $\chi = 1$ for $r \leq 1$, $\chi$ supported on $r \leq 2$, $\chi(r)$ is decreasing as a function of $r$, and let
\begin{equation}\label{5.12}
\phi(r) = \int_{0}^{r} \chi^{2}(\frac{s}{R}) ds,
\end{equation}
and let
\begin{equation}\label{5.12}
M(t) = \int \phi(x) \frac{x}{|x|} \cdot Im[\overline{P_{\leq N} u} \nabla P_{\leq N} u](t, x) dx.
\end{equation}
By direct computation,
\begin{equation}\label{5.13}
\aligned
\frac{d}{dt} \int \phi(x) \frac{x}{|x|} \cdot Im[\bar{u} \nabla u](t,x) dx = 2 \int \chi^{2}(\frac{x}{R}) |\nabla u(t,x)|^{2} dx - \int \chi^{2}(\frac{x}{R}) |u(t,x)|^{4} dx \\
- \frac{1}{2} \int [\frac{1}{|x|} \phi(x) - \chi^{2}(\frac{x}{R})] |u(t,x)|^{4} dx + 2 \int [\frac{1}{|x|} \phi(x) - \chi^{2}(\frac{x}{R})] \frac{x_{j} x_{k}}{|x|^{2}} Re(\overline{\partial_{j} u} \partial_{k} u)(t,x) dx \\
+ O(\frac{1}{R^{2}} \int_{|x| > R} |u(t,x)|^{2} dx).
\endaligned
\end{equation}
\begin{remark}
The last estimate follows from the fact that $\frac{1}{|x|} \phi(x) = 1$ for $|x| \leq R$.
\end{remark}
For $R \gg \lambda(t)$, say $R = T^{1/25}$, since $Q$ is rapidly decreasing,
\begin{equation}\label{5.14}
\int_{|x| > R} \frac{1}{R^{2}} |u(t,x)|^{2} dx \lesssim \frac{1}{R^{2}} \| \epsilon(t) \|_{L^{2}}^{2} + \frac{1}{R^{2} T^{10}} \ll \frac{1}{\lambda(t)^{2}} \| \epsilon(t) \|_{L^{2}}^{2} + \frac{1}{\lambda(t)^{2} T^{10}}.
\end{equation}
Next,
\begin{equation}\label{5.15}
2 \int [\frac{1}{|x|} \phi(x) - \chi^{2}(\frac{x}{R})] \frac{x_{j} x_{k}}{|x|^{2}} Re(\overline{\partial_{j} u} \partial_{k} u)(t,x) dx \geq 0.
\end{equation}
Next, since $Q$ is rapidly decreasing,
\begin{equation}\label{5.16}
- \frac{1}{2} \int [\frac{1}{|x|} \phi(x) - \chi^{2}(\frac{x}{R})] |u(t,x)|^{4} dx \lesssim \frac{1}{\lambda(t)^{2} T^{10}} + \| \epsilon(t) \|_{L^{4}}^{4}.
\end{equation}
Finally,
\begin{equation}\label{5.17}
\aligned
\int \chi^{2}(\frac{x}{R}) |u(t,x)|^{4} dx = \int \chi^{4}(\frac{x}{R}) |u(t,x)|^{4} dx + \int \chi^{2}(\frac{x}{R}) |u(t,x)|^{4} - \int \chi^{4}(\frac{x}{R}) |u(t,x)|^{4} \\ = \int \chi^{4}(\frac{x}{R}) |u(t,x)|^{4} dx + O(\frac{1}{\lambda(t)^{2}} \frac{1}{T^{10}}) + O(\| \epsilon(t) \|_{L^{4}}^{4}).
\endaligned
\end{equation}
Therefore,
\begin{equation}\label{5.18}
\int_{a}^{b} E[\chi(\frac{x}{R}) u(t, x)] dt \leq \int \phi(x) \frac{x}{|x|} \cdot Im[\bar{u} \nabla u](t,x)|_{a}^{b} + o(\int_{a}^{b} \frac{1}{\lambda(t)^{2}} \| \epsilon(t) \|_{L^{2}}^{2} dt) + \int_{a}^{b} \frac{1}{\lambda(t)^{2}} \frac{1}{T^{10}} dt.
\end{equation}
The last estimate uses the Strichartz estimate
\begin{equation}\label{5.19}
\int_{a}^{b} \| \epsilon(t) \|_{L^{4}}^{4} dt \ll \int_{a}^{b} \frac{1}{\lambda(t)^{2}} \| \epsilon(t) \|_{L^{2}}^{2} dt.
\end{equation}
Using the energy lower bound, $E[\chi(\frac{x}{R}) u] \geq \frac{1}{2 \lambda(t)^{2}} \| \epsilon(t) \|_{L^{2}}^{2} - O(\frac{1}{\lambda(t)^{2} T^{10}})$, then
\begin{equation}\label{5.20}
\frac{1}{4} \int_{a}^{b} \frac{1}{\lambda(t)^{2}} \| \epsilon(t) \|_{L^{2}}^{2} dt \leq \int \phi(x) \frac{x}{|x|} \cdot Im[\bar{u} \nabla u](t,x)|_{a}^{b} + O(\frac{1}{T^{9}}).
\end{equation}
Replacing $u$ by $P_{\leq N} u$ and ignoring the error terms arising from the frequency truncation, (which are handled by Proposition $\ref{p5.2}$),
\begin{equation}\label{5.21}
\aligned
\int \phi(x) \frac{x}{|x|} \cdot Im[\overline{P_{\leq N} u} \nabla P_{\leq N} u](t,x) dx = 2 \int Im[\epsilon(t,x) (Q(x) + x \cdot \nabla Q(x))] dx \\ + \int \phi(x) \frac{x}{|x|} \cdot Im[\overline{P_{\leq N} \frac{1}{\lambda(t)} \epsilon(t, \frac{x}{\lambda(t)})} \nabla P_{\leq N} \frac{1}{\lambda(t)} \epsilon(t, \frac{x}{\lambda(t)})] dx.
\endaligned
\end{equation}
Now then, using $(\ref{5.4})$ and Proposition $\ref{p5.3}$,
\begin{equation}\label{5.22}
\aligned
\int \phi(x) \frac{x}{|x|} \cdot Im[\overline{P_{\leq N} \frac{1}{\lambda(t)} \epsilon(t, \frac{x}{\lambda(t)})} \nabla P_{\leq N} \frac{1}{\lambda(t)} \epsilon(t, \frac{x}{\lambda(t)})] dx \\ \lesssim R N \| \epsilon(t) \|_{L^{2}} \| \nabla P_{\leq N} \epsilon(t) \|_{L^{2}} \lesssim R T^{1/100} E[P_{\leq N} u(t)] \lesssim \frac{1}{T^{9}} + \frac{N^{2} R T^{1/100}}{T} \int_{a}^{b} \frac{1}{\lambda(t)^{2}} \| \epsilon(t) \|_{L^{2}}^{2} dt.
\endaligned
\end{equation}
Absorbing the second term on the right hand side of $(\ref{5.22})$ into the left hand side of $(\ref{5.20})$ proves Proposition $\ref{p5.1}$, if Propositions $\ref{p5.2}$ and $\ref{p5.3}$ hold.
\end{proof}

\begin{proof}[Proof of Proposition $\ref{p5.3}$]
By the intermediate value theorem, there exists $t_{0} \in [a, b]$ such that
\begin{equation}\label{5.23}
\| \epsilon(t_{0}) \|_{L^{2}}^{2} \leq \frac{1}{T} \int_{a}^{b} \frac{1}{\lambda(t)^{2}} \| \epsilon(t) \|_{L^{2}}^{2} dt.
\end{equation}
By the Sobolev embedding theorem,
\begin{equation}\label{5.24}
E(P_{\leq N} u(t_{0})) \lesssim \frac{N^{2}}{T} \int_{a}^{b} \frac{1}{\lambda(t)^{2}} \| \epsilon(t) \|_{L^{2}}^{2} dt.
\end{equation}
Computing the change of energy,
\begin{equation}\label{5.25}
\frac{d}{dt} E(P_{\leq N} u(t)) = (-P_{> N}(|u|^{2} u) + [|u|^{2} u - |P_{\leq N} u|^{2} (P_{\leq N} u)], P_{\leq N} u_{t})_{L^{2}}.
\end{equation}
By Proposition $\ref{p5.2}$ and the properties of the Littlewood--Paley projection operator,
\begin{equation}\label{5.26}
\aligned
\int_{a}^{b} |(-P_{> N} (|u|^{2} u), P_{\leq N} u_{t})_{L^{2}}| dt \lesssim N^{2} \| P_{> N} (|u|^{2} u) \|_{L_{t}^{2} L_{x}^{1}} \| P_{\frac{N}{2} < \cdot \leq N} u \|_{L_{t}^{2} L_{x}^{\infty}} \\ + \| P_{> N} (|u|^{2} u) \|_{L_{t}^{2} L_{x}^{1}} \| P_{> \frac{N}{8}} u \|_{L_{t}^{2} L_{x}^{\infty}} \| P_{\leq N} u \|_{L_{t,x}^{\infty}}^{2} + \| P_{N < \cdot < 2N} (|u|^{2} u) \|_{L_{t,x}^{2}} \| P_{\frac{N}{2} < \cdot N} (|P_{> \frac{N}{8}} u|^{2} (P_{> \frac{N}{8}} u)) \|_{L_{t,x}^{2}}  \\ \lesssim \frac{N^{2}}{T} + \frac{N^{2}}{T} \int_{a}^{b} \| \epsilon(t) \|_{L^{2}}^{2} \frac{1}{\lambda(t)^{2}} dt.
\endaligned
\end{equation}
Now decompose
\begin{equation}\label{5.27}
|u|^{2} u - |P_{\leq N} u|^{2} (P_{\leq N} u) = (P_{> N} u)^{3} + 3 (P_{> N} u)^{2} (P_{\leq N} u) + 3 (P_{> N} u)(P_{\leq N} u)^{2}.
\end{equation}
Now then,
\begin{equation}\label{5.28}
\int |((P_{> N} u)^{2} u, \Delta P_{\leq N} u)_{L^{2}}| dt \lesssim N^{2} \| P_{> N} u \|_{L_{t}^{2} L_{x}^{\infty}}^{2} \| u \|_{L_{t}^{\infty} L_{x}^{2}}^{2} \lesssim \frac{N^{2}}{T} (\int_{a}^{b} \| \epsilon(t) \|_{L^{2}}^{2} \lambda(t)^{-2} dt) + \frac{N^{2}}{T^{10}}.
\end{equation}
\begin{equation}\label{5.29}
\aligned
\int |((P_{> N} u)^{2} u, P_{\leq N} (|u|^{2} u))_{L^{2}}| dt \lesssim \| P_{N < \cdot < 2N} ((P_{> N} u)^{2} u) \|_{L_{t}^{1} L_{x}^{\infty}} \| u \|_{L_{t}^{\infty} L_{x}^{2}}^{2} \| P_{\leq N} u \|_{L_{t,x}^{\infty}} \\
+ \| P_{\leq 2N}((P_{> N} u)^{2} u) \|_{L_{t,x}^{2}} \| P_{\leq N}((P_{> N} u)^{2} u) \|_{L_{t,x}^{2}} \lesssim \frac{N^{2}}{T} (\int_{a}^{b} \| \epsilon(t) \|_{L^{2}}^{2} \lambda(t)^{-2} dt) + \frac{N^{2}}{T^{10}}. 
\endaligned
\end{equation}

\begin{equation}\label{5.30}
\int |((P_{> N} u)(P_{\leq N} u)^{2}, \Delta P_{\leq N} u)_{L^{2}}| dt \lesssim N^{2} \| P_{> N} u \|_{L_{t}^{2} L_{x}^{\infty}} \| P_{> \frac{N}{8}} u \|_{L_{t}^{2} L_{x}^{\infty}} \| u \|_{L_{t}^{\infty} L_{x}^{2}}^{2} \lesssim \frac{N^{2}}{T} (\int_{a}^{b} \| \epsilon(t) \|_{L^{2}}^{2} \lambda(t)^{-2} dt) + \frac{N^{2}}{T^{10}}. 
\end{equation}
\begin{equation}\label{5.31}
\aligned
\int |((P_{> N} u)(P_{\leq N} u)^{2}, P_{\leq N}((P_{> N} u) u^{2}))_{L^{2}}| dt \lesssim \| P_{\leq N} u \|_{L_{t,x}^{\infty}}^{2} \| P_{> N} u \|_{L_{t}^{2} L_{x}^{\infty}}^{2} \| u \|_{L_{t}^{\infty} L_{x}^{2}}^{2} \\\lesssim \frac{N^{2}}{T} (\int_{a}^{b} \| \epsilon(t) \|_{L^{2}}^{2} \lambda(t)^{-2} dt) + \frac{N^{2}}{T^{10}}.
\endaligned
\end{equation}
\begin{equation}\label{5.32}
\aligned
\int |((P_{> N} u)(P_{\leq N} u)^{2}, P_{\leq N}((P_{\leq N} u)^{3}))_{L^{2}}| dt \lesssim \| P_{\leq N} u \|_{L_{t,x}^{\infty}}^{2} \| P_{> \frac{N}{8}} u \|_{L_{t}^{2} L_{x}^{\infty}}^{2} \| u \|_{L_{t}^{\infty} L_{x}^{2}}^{2} \\\lesssim \frac{N^{2}}{T} (\int_{a}^{b} \| \epsilon(t) \|_{L^{2}}^{2} \lambda(t)^{-2} dt) + \frac{N^{2}}{T^{10}}.
\endaligned
\end{equation}
This completes the proof of Proposition $\ref{p5.3}$.
\end{proof}

\begin{proof}[Proof of Proposition $\ref{p5.2}$]
This proposition is proved using induction on frequency. Fix $T^{1/6} \leq M \leq T^{1/3}$.
\begin{equation}\label{5.33}
\| P_{> M} u \|_{U_{\Delta}^{2}([a, b] \times \mathbb{R}^{2})} \lesssim \frac{1}{T^{5}} + \inf_{t \in [a, b]} \| \epsilon(t) \|_{L^{2}} + \| P_{> M} (|u|^{2} u) \|_{DU_{\Delta}^{2}([a, b] \times \mathbb{R}^{2})}.
\end{equation}
Since $u = \frac{1}{\lambda} Q(\frac{x}{\lambda}) + \frac{1}{\lambda} \epsilon(t, \frac{x}{\lambda})$, for any $\delta > 0$,
\begin{equation}\label{5.34}
\aligned
\| P_{> M}(|u|^{2} u) \|_{DU_{\Delta}^{2}([a, b] \times \mathbb{R}^{2})} \lesssim_{\delta} \| (P_{> \frac{M}{8}} u) \|_{L_{t}^{\infty} L_{x}^{2}} \| \frac{1}{\lambda} \epsilon(t, \frac{x}{\lambda}) \|_{L_{t}^{2} L_{x}^{\infty}}^{2} \\
+ \| (P_{> \frac{M}{8}} u) (\frac{1}{\lambda(t)} Q(\frac{x}{\lambda(t)})^{2}) \|_{L_{t}^{1} L_{x}^{2}([a, b] \times \mathbb{R}^{2})}^{\delta} \cdot (M^{-1/2} \sum_{j} \| (P_{> \frac{M}{8}} u) (\frac{1}{\lambda(t)} Q(\frac{x}{\lambda(t)})^{2}) \|_{L_{t, x}^{2} ([a, b] \times \{ |x| \sim 2^{j} \})})^{1 - \delta}.
\endaligned
\end{equation}
Using standard Strichartz estimates,
\begin{equation}\label{5.35}
\| (P_{> \frac{M}{8}} u) \|_{L_{t}^{\infty} L_{x}^{2}} \| \frac{1}{\lambda} \epsilon(t, \frac{x}{\lambda}) \|_{L_{t}^{2} L_{x}^{\infty}}^{2} \lesssim \| P_{> \frac{M}{8}} u \|_{U_{\Delta}^{2}([a, b] \times \mathbb{R}^{2})} (\int_{a}^{b} \frac{1}{\lambda(t)^{2}} \| \epsilon(t) \|_{L^{2}}^{2} dt) \lesssim \eta_{\ast}^{1/2} \| P_{> \frac{M}{8}} u \|_{U_{\Delta}^{2}([a, b] \times \mathbb{R}^{2})}.
\end{equation}
Since $\lambda(t) \geq 1$,
\begin{equation}\label{5.36}
\aligned
\| (P_{> \frac{M}{8}} u)(\frac{1}{\lambda(t)} Q(\frac{x}{\lambda(t)})) \|_{L_{t,x}^{2}} \lesssim \sup_{R > 0} R^{-1/2} \| P_{> \frac{M}{8}} u \|_{L_{t,x}^{2}(|x| \leq R)} \cdot \sum_{j \geq 0} 2^{-j} \sum_{k} 2^{k/2} Q(2^{-j} 2^{k}) \\ \lesssim \frac{1}{M^{1/2}} \| P_{> \frac{M}{8}} u \|_{U_{\Delta}^{2}([a, b] \times \mathbb{R}^{2})}.
\endaligned
\end{equation}
Also, since $\lambda(t) \geq 1$,
\begin{equation}\label{5.37}
\sum_{k} 2^{k/2} \| \frac{1}{\lambda(t)} Q(\frac{x}{\lambda(t)}) \|_{L_{t,x}^{\infty}(|x| \sim 2^{k})} \lesssim \sum_{j \geq 0} 2^{-j} \sum_{k} 2^{k/2} Q(2^{-j} 2^{k}) \lesssim 1.
\end{equation}
Finally, make the trivial estimate
\begin{equation}\label{5.38}
\| \frac{1}{\lambda(t)} Q(\frac{x}{\lambda(t)}) \|_{L_{t}^{2} L_{x}^{\infty}} \lesssim (\int_{a}^{b} \frac{1}{\lambda(t)^{2}} dt)^{1/2} \lesssim T^{1/2}.
\end{equation}
Plugging $(\ref{5.35})$--$(\ref{5.38})$ into $(\ref{5.34})$,
\begin{equation}\label{5.39}
(\ref{5.34}) \lesssim (\eta_{\ast}^{1/2} + \frac{T^{\delta/2}}{M^{1 - \delta/2}}) \| P_{> \frac{M}{8}} u \|_{U_{\Delta}^{2}([a, b] \times \mathbb{R}^{2})}.
\end{equation}
Arguing by induction on frequency, starting from $(\ref{5.2})$, which implies
\begin{equation}\label{5.40}
\| u \|_{U_{\Delta}^{2}([a, b] \times \mathbb{R}^{2})}^{2} \lesssim T,
\end{equation}
for $N = T^{1/3}$, there exists some $c > 0$ such that
\begin{equation}\label{5.41}
\| P_{> N} u \|_{U_{\Delta}^{2}([a, b] \times \mathbb{R}^{2})}^{2} \lesssim \frac{1}{T^{10}} + \frac{1}{T} \int_{a}^{b} \| \epsilon(t) \|_{L^{2}}^{2} \frac{1}{\lambda(t)^{2}} dt + T(\eta_{\ast}^{\frac{c}{2} \cdot \ln(T)} + T^{-c \ln(T)}) \lesssim \frac{1}{T^{10}} + \frac{1}{T} \int_{a}^{b} \| \epsilon(t) \|_{L^{2}}^{2} \frac{1}{\lambda(t)^{2}} dt.
\end{equation}
\end{proof}

\subsection{Estimates for the Chern--Simons--Schr{\"o}dinger equation}
Now we can prove a similar estimate for a solution to the Chern--Simons--Schr{\"o}dinger equation.
\begin{proposition}\label{p5.4}
Suppose $u$ is a solution to $(\ref{1.1})$,
\begin{equation}\label{5.42}
i u_{t} + \Delta u = \frac{2m}{r^{2}} A_{\theta}[u] u + A_{0}[u] u + \frac{1}{r^{2}} A_{\theta}[u]^{2} u - g |u|^{2} u, \qquad \| u \|_{L^{2}} = \| Q \|_{L^{2}}.
\end{equation}
Furthermore, suppose that for some interval $[a, b]$ with $a > 0$,
\begin{equation}\label{5.43}
\int_{a}^{b} \lambda(t)^{-2} dt = T, \qquad T > \eta_{\ast}^{-1},
\end{equation}
\begin{equation}\label{5.44}
\sup_{t > 0} dist(u(t), \mathcal M) \leq \eta_{\ast},
\end{equation}
and that $u$ is equivariant of order $m$. Finally, suppose that for all $t \in [a, b]$,
\begin{equation}\label{5.45}
1 \leq \lambda(t) \leq T^{1/100}.
\end{equation}
Then,
\begin{equation}\label{5.46}
\int_{a}^{b} \| \epsilon(t) \|_{L^{2}}^{2} \lambda(t)^{-2} dt \leq 3(\epsilon_{2}(a), Q + x \cdot \nabla Q)_{L^{2}} - (3 \epsilon_{2}(b), Q + x \cdot \nabla Q)_{L^{2}} + O(T^{-8}).
\end{equation} 
\end{proposition}

Proof of Proposition $\ref{p5.4}$ is the same as the proof of Proposition $\ref{p5.1}$, making use of a long time Strichartz estimate that is analogous to Proposition $\ref{p5.2}$ and an almost conservation of energy result analogous to Proposition $\ref{p5.3}$. The main difficulty is the fact that the Chern--Simons--Schr{\"o}dinger equation is nonlocal. Once again, we make a bootstrap assumption analogous to $(\ref{5.6})$.

\begin{proposition}[Long time Strichartz estimate]\label{p5.5}
If $u$ satisfies the conditions of Proposition $\ref{p5.4}$,
\begin{equation}\label{5.47}
\| P_{> N} u \|_{U_{\Delta}^{2}([a, b] \times \mathbb{R}^{2})}^{2} \lesssim \frac{1}{T^{10}} + \frac{1}{T} \int_{a}^{b} \| \epsilon(t) \|_{L^{2}}^{2} \lambda(t)^{-2} dt.
\end{equation}
\end{proposition}
\begin{proof}
As in the proof of Proposition $\ref{p5.2}$, for any $T^{1/6} \leq M \leq T^{1/3}$,
\begin{equation}\label{5.48}
\| P_{> M} (g |u|^{2} u) \|_{DU_{\Delta}^{2}([a, b] \times \mathbb{R}^{2})} \lesssim g (\eta_{\ast}^{1/2} + \frac{T^{\delta/2}}{M^{1 - \delta/2}}) \| P_{> \frac{M}{8}} u \|_{U_{\Delta}^{2}([a, b] \times \mathbb{R}^{2})}.
\end{equation}
Expanding
\begin{equation}\label{5.49}
\aligned
A_{\theta}[u] = -\frac{1}{2} \int_{0}^{r} |u(t, s)|^{2} s ds = -\frac{1}{2} \int_{0}^{r} \frac{1}{\lambda(t)^{2}} Q(\frac{s}{\lambda(t)})^{2} s ds \\ - Re \int_{0}^{r} \frac{1}{\lambda(t)^{2}} Q(\frac{s}{\lambda(t)}) \overline{\epsilon(t, \frac{s}{\lambda(t)})} s ds
- \frac{1}{2} \int_{0}^{r} \frac{1}{\lambda(t)^{2}} |\epsilon(t, \frac{s}{\lambda(t)})|^{2} s ds.
\endaligned
\end{equation}
Since
\begin{equation}\label{5.50}
\aligned
-\frac{1}{2 r^{2}} \int_{0}^{r} \frac{1}{\lambda(t)^{2}} Q(\frac{s}{\lambda(t)})^{2} s ds \lesssim \inf \{ \frac{1}{\lambda(t)^{2}}, \frac{1}{r^{2} \lambda(t)} \}, \\
-\frac{1}{2 r^{2}} \int_{0}^{r} \frac{1}{\lambda(t)^{2}} |\epsilon(t, \frac{s}{\lambda(t)})|^{2} s ds \lesssim \frac{1}{\lambda(t)^{2}} \| \epsilon(t) \|_{L^{\infty}}^{2},
\endaligned
\end{equation}
then using an argument similar to $(\ref{5.33})$--$(\ref{5.41})$,
\begin{equation}\label{5.51}
\| (\frac{1}{r^{2}} A_{\theta}[u]) (P_{> \frac{M}{8}} u) \|_{DU_{\Delta}^{2}([a, b] \times \mathbb{R}^{2})} \lesssim (\eta_{\ast}^{1/2} + \frac{T^{\delta/2}}{M^{1 - \delta/2}}) \| P_{> \frac{M}{8}} u \|_{U_{\Delta}^{2}([a, b] \times \mathbb{R}^{2})}.
\end{equation}
Meanwhile,
\begin{equation}\label{5.52}
P_{> M}(\frac{1}{r^{2}} A_{\theta}[u] (P_{\leq \frac{M}{8}} u)) = P_{> M}(P_{> \frac{M}{2}}(\frac{1}{r^{2}} A_{\theta}[u]) \cdot (P_{\leq \frac{M}{8}} u)).
\end{equation}

Making a change of variables,
\begin{equation}\label{5.53}
\frac{m}{r^{2}} A_{\theta}[u] = -\frac{m}{2} \int_{0}^{1} |u(t, s r)|^{2} s ds.
\end{equation}
Therefore,
\begin{equation}\label{5.54}
\| P_{> \frac{M}{2}} (\frac{2m}{r^{2}} A_{\theta}[u]) \|_{L_{t,x}^{2}} \lesssim \| P_{> \frac{M}{8}} u \|_{L_{t}^{2} L_{x}^{\infty}}.
\end{equation}
Decomposing $u = \frac{1}{\lambda(t)} Q(\frac{x}{\lambda}) + \frac{1}{\lambda(t)} \epsilon(t, \frac{x}{\lambda(t)})$ and using $(\ref{5.34})$ for the $Q$ term and $(\ref{5.35})$ for the $\epsilon$ term,
\begin{equation}\label{5.55}
\| P_{> \frac{M}{2}} (\frac{2m}{r^{2}} A_{\theta}[u]) (P_{\leq \frac{M}{8}} u) \|_{DU_{\Delta}^{2}} \lesssim (\eta_{\ast}^{1/2} + \frac{T^{\delta/2}}{M^{1 - \delta/2}}) \| P_{> \frac{M}{8}} u \|_{U_{\Delta}^{2}([a, b] \times \mathbb{R}^{2})}.
\end{equation}
By a similar argument,
\begin{equation}\label{5.56}
\| P_{> \frac{M}{2}} (\frac{1}{r^{2}} A_{\theta}[u]^{2}) (P_{\leq \frac{M}{8}} u) \|_{DU_{\Delta}^{2}} \lesssim (\eta_{\ast}^{1/2} + \frac{T^{\delta/2}}{M^{1 - \delta/2}}) \| P_{> \frac{M}{8}} u \|_{U_{\Delta}^{2}([a, b] \times \mathbb{R}^{2})}.
\end{equation}
Indeed, let $\chi \in C_{0}^{\infty}(\mathbb{R}^{2})$, $\chi(x) = 1$ for $|x| \leq 1$. Let
\begin{equation}\label{5.57}
\tilde{A}_{\theta}[u] = \int_{0}^{r} |P_{\leq \frac{M}{8}} u(t, s)|^{2} s ds.
\end{equation}
Then,
\begin{equation}\label{5.58}
\frac{\tilde{A}_{\theta}[u]^{2}}{r^{2}} = r^{2} \int_{0}^{1} |P_{\leq \frac{M}{8}} u(t, s r)|^{2} s ds \cdot \int_{0}^{1} |P_{\leq \frac{M}{8}} u(t, s' r)|^{2} s' ds'.
\end{equation}
Then by Fourier support arguments, for any $N$,
\begin{equation}\label{5.59}
\| P_{> M}(\frac{\tilde{A}_{0}[u]^{2}}{r^{2}}) \|_{L_{t,x}^{2}} \lesssim \| P_{> \frac{M}{2}}(\chi(r) r^{2}) \|_{L_{t}^{\infty} L_{x}^{2}} \| P_{\leq \frac{M}{8}} u \|_{L_{t}^{8} L_{x}^{\infty}}^{4} \lesssim \frac{1}{M^{N}} T^{1/2} M^{3}.
\end{equation}
Now let $\psi_{j}(r) = \chi(2^{-j} r) - \chi(2^{-j + 1} r)$. Then,
\begin{equation}\label{5.60}
\sum_{j} \| P_{> \frac{M}{2}}(\psi_{j}(r) r^{2}) \|_{L_{x}^{2}} \| P_{\leq \frac{M}{8}} u \|_{L_{t}^{8} L_{x}^{\infty}}^{4} \lesssim \sum_{j} 2^{-j} \frac{1}{M^{N}} T^{1/2} M^{3} \lesssim \frac{T^{1/2}}{M^{N}}.
\end{equation}
Next, for any $R$,
\begin{equation}\label{5.61}
\| \frac{1}{r^{2}} \int_{0}^{r} |P_{> \frac{M}{8}} u(t, s)| |u(t, s)| s ds \|_{L^{2}(R \leq \cdot \leq 2R)} \lesssim \sum_{j \leq 0} 2^{j} \| P_{> \frac{M}{8}} u \|_{L^{\infty}} \| u \|_{L^{2}(2^{j} R \leq \cdot \leq 2^{j + 1} R)}.
\end{equation}
Therefore, by Young's inequality, $\| A_{0}[u] \|_{L^{\infty}} + \| \tilde{A}_{0}[u] \|_{L^{\infty}} \lesssim \| u \|_{L^{2}}^{2}$, $(\ref{5.60})$, $(\ref{5.61})$, $M \geq T^{1/6}$, and again using $(\ref{5.34})$ for the $Q$ term and $(\ref{5.35})$ for the $\epsilon$ term,
\begin{equation}\label{5.62}
\| P_{> M}(\frac{A_{0}[u]^{2}}{r^{2}}) u \|_{DU_{\Delta}^{2}([a, b] \times \mathbb{R}^{2})} (\eta_{\ast}^{1/2} + \frac{T^{\delta/2}}{M^{1 - \delta/2}}) \| P_{> \frac{M}{8}} u \|_{U_{\Delta}^{2}([a, b] \times \mathbb{R}^{2})} + \frac{1}{T^{10}}.
\end{equation}

Finally turn to the $A_{0}[u] u$ term. Since $Q \in \mathcal H_{m}^{1}$,
\begin{equation}\label{5.63}
\| \frac{m}{r} Q \|_{L^{2}} \sim \int_{0}^{\infty} \frac{m^{2}}{r} Q(r)^{2} dr < \infty.
\end{equation}
Therefore,
\begin{equation}\label{5.64}
\int_{r}^{\infty} \frac{1}{\lambda(t)^{2}} Q(\frac{s}{\lambda(t)})^{2} \frac{1}{s} ds \lesssim \inf \{ \frac{1}{\lambda(t) r^{2}}, \frac{1}{\lambda(t)} \}.
\end{equation}
As in $(\ref{5.61})$,
\begin{equation}\label{5.65}
\| \int_{r}^{\infty} \frac{1}{\lambda(t)^{2}} |\epsilon(t, \frac{s}{\lambda(t)})|^{2} \frac{1}{s} ds \|_{L^{2}(R \leq \cdot \leq 2R)} \lesssim \frac{1}{\lambda(t)} \sum_{j \geq 0} 2^{-j/2} \| \frac{1}{\lambda(t)} \epsilon(t, \frac{s}{\lambda(t)}) \|_{L^{2}(2^{j} R \leq \cdot \leq 2^{j + 1} R)} \| \epsilon \|_{L^{\infty}}.
\end{equation}
Again, as in $(\ref{5.33})$--$(\ref{5.41})$,
\begin{equation}\label{5.66}
\| A_{0}[u] (P_{> \frac{M}{8}} u) \|_{DU_{\Delta}^{2}([a, b] \times \mathbb{R}^{2})} \lesssim (\eta_{\ast}^{1/2} + \frac{T^{\delta/2}}{M^{1 - \delta/2}}) \| P_{> \frac{M}{8}} u \|_{U_{\Delta}^{2}([a, b] \times \mathbb{R}^{2})}.
\end{equation}

It remains to estimate
\begin{equation}\label{5.67}
P_{> M}(A_{0}[u] (P_{\leq \frac{M}{8}} u)).
\end{equation}
Let
\begin{equation}\label{5.68}
\tilde{A}_{0}[u] = -\int_{r}^{\infty} (m + \tilde{A}_{\theta}[u])(t, s) |P_{\leq \frac{M}{8}} u(t, s)|^{2} \frac{ds}{s}.
\end{equation}
By direct computation, following $(\ref{5.65})$ and using the fact that $\| A_{\theta}[u] \|_{L^{\infty}} + \| \tilde{A}_{\theta}[u] \|_{L^{\infty}} \lesssim \| u \|_{L^{2}}^{2}$ and $|A_{\theta}[u] - \tilde{A}_{\theta}[u]|(r) \lesssim |\int_{0}^{r} |u| |P_{\geq \frac{M}{8}} u| s ds| \lesssim r \| u \|_{L^{2}} \| P_{\geq \frac{M}{8}} u \|_{L^{\infty}}$,
\begin{equation}\label{5.69}
\aligned
\| A_{0}[u] - \tilde{A}_{0}[u] \|_{L_{t,x}^{2}} \lesssim \| \int_{r}^{\infty} \frac{(m + \tilde{A}_{\theta}[u](t, s)}{s} |u(t, s)| |P_{> \frac{M}{8}} u(t, s)| ds \|_{L_{t,x}^{2}} + \| \int_{r}^{\infty} \frac{|A_{\theta}[u] - \tilde{A}_{\theta}[u]}{s} |u(t, s)|^{2} ds \|_{L_{t,x}^{2}} \\
\lesssim \| u \|_{L^{2}}^{2} \| \int_{r}^{\infty} \frac{|u(t, s)| |P_{\geq \frac{M}{8}} u|}{s} ds \|_{L_{t,x}^{2}} + \| \int_{r}^{\infty} \| u \|_{L^{2}} \| P_{\geq \frac{M}{8}} u \|_{L^{\infty}} |u(t, s)|^{2} ds \|_{L_{t,x}^{2}} \lesssim \| u \|_{L_{t}^{\infty} L_{x}^{2}}^{3} \| P_{\geq \frac{M}{8}} u \|_{L_{t}^{2} L_{x}^{\infty}}.
\endaligned
\end{equation}
By $(\ref{5.69})$, again using $(\ref{5.34})$ for the $Q$ term and $(\ref{5.35})$ for the $\epsilon$ term,
\begin{equation}\label{5.70}
\| (A_{0}[u] - \tilde{A}_{0}[u]) (P_{\leq \frac{M}{8}} u) \|_{DU_{\Delta}^{2}([a, b] \times \mathbb{R}^{2})}\lesssim (\eta_{\ast}^{1/2} + \frac{T^{\delta/2}}{M^{1 - \delta/2}}) \| P_{> \frac{M}{8}} u \|_{U_{\Delta}^{2}([a, b] \times \mathbb{R}^{2})}.
\end{equation}
Letting
\begin{equation}\label{5.71}
c = -\int_{0}^{\infty} \frac{(m + \tilde{A}_{\theta}[u](t, s)}{s} |P_{\leq \frac{M}{8}} u(t, s)|^{2} ds,
\end{equation}
\begin{equation}\label{5.72}
P_{> M} (\tilde{A}_{0}[u] (P_{\leq \frac{M}{8}} u)) = P_{> M} ((\tilde{A}_{0}[u] - c) \cdot (P_{\leq \frac{M}{8}} u)) = P_{> M} (-\int_{0}^{r} \frac{m + \tilde{A}_{\theta}[u](t, s)}{s} |P_{|\leq \frac{M}{8}} u(t, s)|^{2} ds \cdot (P_{\leq \frac{M}{8}} u(t,r))).
\end{equation}
Expanding out $(\ref{5.72})$, using $(\ref{5.57})$,
\begin{equation}\label{5.73}
\aligned
(\ref{5.72}) = -\int_{0}^{1} m |P_{\leq \frac{M}{8}} u(t, sr)|^{2} \frac{ds}{s} \cdot (P_{\leq \frac{M}{8}} u(t, r)) - \int_{0}^{1} \frac{\tilde{A}_{\theta}(t, sr)}{s^{2} r^{2}} |u(t, sr)|^{2} s ds \cdot (P_{\leq \frac{M}{8}} u(t, r)) \\
= -\int_{0}^{1} m |P_{\leq \frac{M}{8}} u(t, sr)|^{2} \frac{ds}{s} \cdot (P_{\leq \frac{M}{8}} u(t, r)) \\ - \int_{0}^{1} \int_{0}^{1} |P_{\leq \frac{M}{8}} u(t, s s' r)|^{2} |P_{\leq \frac{M}{8}} u(t, sr)|^{2} s s' ds ds' \cdot (P_{\leq \frac{M}{8}} u(t, r)).
\endaligned
\end{equation}
Therefore, $P_{> M} (\ref{5.72}) = 0$. Arguing using induction on frequency proves Proposition $\ref{p5.5}$.
\end{proof}

\begin{remark}
In fact, the same argument implies
\begin{equation}\label{5.74}
\| P_{> N} u \|_{U_{\Delta}^{2}([a, b] \times \mathbb{R}^{2})}^{2} \lesssim \frac{1}{T^{10}} + \inf_{t \in [a, b]} \| \epsilon(t) \|_{L^{2}}^{2}.
\end{equation}
Of course, by the intermediate value theorem,
\begin{equation}\label{5.75}
\inf_{t \in [a, b]} \| \epsilon(t) \|_{L^{2}}^{2} \lesssim \frac{1}{T} \int_{a}^{b} \| \epsilon(t) \|_{L^{2}}^{2} \lambda(t)^{-2} dt,
\end{equation}
but to prove rigidity we will use a better bound.
\end{remark}

\begin{proposition}[Energy estimate]\label{p5.6}
If $u$ satisfies the conditions of Proposition $\ref{p5.4}$, and $N \geq T^{1/3}$,
\begin{equation}\label{5.76}
\sup_{t \in [a, b]} E(P_{\leq N} u)(t) \lesssim N^{2} \inf_{t \in [a, b]} \| \epsilon(t) \|_{L^{2}}^{2} + \frac{N^{2}}{T^{10}}.
\end{equation}
In particular,
\begin{equation}\label{5.77}
\sup_{t \in [a, b]} E(P_{\leq N} u)(t) \lesssim \frac{N^{2}}{T} \int_{a}^{b} \| \epsilon(t) \|_{L^{2}}^{2} \lambda(t)^{-2} dt + \frac{N^{2}}{T^{10}}.
\end{equation}
\end{proposition}
\begin{proof}
Recall that
\begin{equation}\label{5.78}
E[u] = \frac{1}{2} \int |\partial_{r} u|^{2} + \frac{1}{2} \int (\frac{m + A_{\theta}[u]}{r})^{2} |u|^{2} - \frac{g}{4} \int |u|^{4}.
\end{equation}
Now compute
\begin{equation}\label{5.79}
\frac{d}{dt} \frac{1}{2} \int |\partial_{r} P_{\leq N} u|^{2} = Re \int (\overline{\partial_{r} P_{\leq N} u_{t}})(\partial_{r} P_{\leq N} u) r dr = -Re \int (\overline{P_{\leq N} u_{t}})(\partial_{rr} + \frac{1}{r} \partial_{r})(P_{\leq N} u) r dr.
\end{equation}
\begin{equation}\label{5.80}
\frac{d}{dt}(-\frac{g}{4} \int |P_{\leq N} u|^{4}) = -g Re \int (\overline{P_{\leq N} u_{t}})(|P_{\leq N} u|^{2} (P_{\leq N} u)) dx.
\end{equation}
\begin{equation}\label{5.81}
\aligned
\frac{d}{dt} \frac{1}{2} \int (\frac{m + A_{\theta}[P_{\leq N} u]}{r})^{2} |P_{\leq N} u|^{2} = Re \int (\frac{m + A_{\theta}[P_{\leq N} u]}{r})^{2} (\overline{P_{\leq N} u_{t}})(P_{\leq N} u) \\ + Re \int \frac{\dot{A}_{\theta}[P_{\leq N} u]}{r} (\frac{m + A_{\theta}[P_{\leq N} u]}{r}) |P_{\leq N} u|^{2}.
\endaligned
\end{equation}

Split
\begin{equation}\label{5.82}
|P_{\leq N} u|^{2} (P_{\leq N} u) = P_{\leq N}(|u|^{2} u) + P_{> N}(|u|^{2} u) + |P_{\leq N} u|^{2} (P_{\leq N} u) - |u|^{2} u.
\end{equation}
\begin{equation}\label{5.83}
\frac{2m A_{\theta}[P_{\leq N} u]}{r^{2}} (P_{\leq N} u) = P_{\leq N} (\frac{2m A_{\theta}[u]}{r^{2}} u) + P_{> N} (\frac{2m}{r^{2}} A_{\theta}[u] u) + \frac{2m A_{\theta}[P_{\leq N} u]}{r^{2}} (P_{\leq N} u) - \frac{2m A_{\theta}[u]}{r^{2}} u.
\end{equation}
\begin{equation}\label{5.84}
\frac{A_{\theta}[P_{\leq N} u]^{2}}{r^{2}} (P_{\leq N} u) = P_{\leq N} (\frac{A_{\theta}[u]^{2}}{r^{2}} u) + P_{> N}(\frac{A_{\theta}[u]^{2}}{r^{2}} u) + \frac{A_{\theta}[P_{\leq N} u]^{2}}{r^{2}} (P_{\leq N} u) - \frac{A_{\theta}[u]^{2}}{r^{2}} u.
\end{equation}
Finally,
\begin{equation}\label{5.85}
Re \int \frac{\dot{A}_{\theta}[P_{\leq N} u]}{r} (\frac{m + A_{\theta}[P_{\leq N} u]}{r}) |P_{\leq N} u|^{2} = \int (\overline{P_{\leq N} u_{t}}) A_{0}[P_{\leq N} u](P_{\leq N} u).
\end{equation}
Expanding,
\begin{equation}\label{5.86}
A_{0}[P_{\leq N} u](P_{\leq N} u) = P_{\leq N}(A_{0}[u] u) + P_{> N}(A_{0}[u] u) + A_{0}[P_{\leq N} u](P_{\leq N} u) - A_{0}[u] u.
\end{equation}
Since $u$ is $m$-equivariant,
\begin{equation}\label{5.87}
(\partial_{rr} + \frac{1}{r} \partial_{r}) (P_{\leq N} u) - \frac{m^{2}}{r^{2}} (P_{\leq N} u) = \Delta P_{\leq N} u = P_{\leq N} \Delta u = P_{\leq N} (\partial_{rr} + \frac{1}{r} \partial_{r} - \frac{m^{2}}{r^{2}}) u.
\end{equation}
Taking the first term in $(\ref{5.79})$--$(\ref{5.87})$,
\begin{equation}\label{5.88}
Re \int (\overline{P_{\leq N} u_{t}}) \cdot P_{\leq N}(-\partial_{rr} u - \frac{1}{r} \partial_{r} u + \frac{m^{2}}{r^{2}} u + \frac{2m A_{\theta}[u]}{r^{2}} u + \frac{A_{\theta}[u]^{2}}{r^{2}} u + A_{0}[u] u - g |u|^{2} u) = Re \int (\overline{P_{\leq N} u_{t}}) (P_{\leq N} i u_{t}) = 0.
\end{equation}
Next, taking the second terms in $(\ref{5.79})$--$(\ref{5.87})$,
\begin{equation}\label{5.89}
\aligned
Re \int (\overline{P_{\leq N} u_{t}}) \cdot P_{> N} (-g |u|^{2} u + \frac{2m}{r^{2}} A_{\theta}[u] u + \frac{A_{\theta}[u]^{2}}{r^{2}} u + A_{0}[u] u) \\
= Re \int (\overline{P_{\frac{N}{2} \leq \cdot \leq N} u_{t}}) \cdot P_{N < \cdot < 2N} (-g |u|^{2} u + \frac{2m}{r^{2}} A_{\theta}[u] u + \frac{A_{\theta}[u]^{2}}{r^{2}} u + A_{0}[u] u).
\endaligned
\end{equation}
Following the computations in the proof of Proposition $\ref{p5.5}$,
\begin{equation}\label{5.90}
\aligned
(\ref{5.89}) \lesssim N^{2} \| P_{> \frac{N}{2}} u \|_{L_{t}^{2} L_{x}^{\infty}} \| P_{> N} (-g |u|^{2} u + \frac{2m}{r^{2}} A_{\theta}[u] u + \frac{A_{\theta}[u]^{2}}{r^{2}} u + A_{0}[u] u) \|_{L_{t}^{2} L_{x}^{1}} \\
+ \| P_{\frac{N}{2} < \cdot < 2N} (-g |u|^{2} u + \frac{2m}{r^{2}} A_{\theta}[u] u + \frac{A_{\theta}[u]^{2}}{r^{2}} u + A_{0}[u] u) \|_{L_{t,x}^{2}}^{2} \lesssim N^{2} \inf_{t \in [a, b]} \| \epsilon(t) \|_{L^{2}}^{2} + \frac{N^{2}}{T^{10}}.
\endaligned
\end{equation}

Now take the third term in $(\ref{5.82})$--$(\ref{5.86})$. To simplify notation let
\begin{equation}\label{5.91}
\aligned
\mathcal N = -g (|u|^{2} u - |P_{\leq N} u|^{2} (P_{\leq N} u)) + \frac{2m}{r^{2}} (A_{\theta}[P_{\leq N} u](P_{\leq N} u) - A_{\theta}[u] u) \\ + \frac{1}{r^{2}}(A_{\theta}[P_{\leq N} u]^{2} (P_{\leq N} u) - A_{\theta}[u]^{2} u) + (A_{0}[P_{\leq N} u](P_{\leq N} u) - A_{0}[u] u).
\endaligned
\end{equation}
Since the computations proving Proposition $\ref{p5.5}$ still hold,
\begin{equation}\label{5.92}
\| \mathcal N \|_{L_{t}^{2} L_{x}^{1}}^{2} \lesssim \inf_{t \in [a, b]} \| \epsilon(t) \|_{L^{2}}^{2} + \frac{1}{T^{10}}.
\end{equation}

Now split,
\begin{equation}\label{5.93}
\aligned
\int (\overline{P_{\leq N} u_{t}}) \mathcal N = \int (\overline{P_{\frac{N}{8} < \cdot < N} u_{t}})\mathcal N 
+ \int (\overline{P_{\leq \frac{N}{8}} u_{t}})\mathcal N.
\endaligned
\end{equation}
Again following the computations in $(\ref{5.90})$,
\begin{equation}\label{5.94}
\int (\overline{P_{\frac{N}{8} < \cdot < N} u_{t}})\mathcal N \lesssim \frac{N^{2}}{T^{10}} + N^{2} \inf_{t \in [a, b]} \| \epsilon(t) \|_{L^{2}}^{2}.
\end{equation}

By the Sobolev embedding theorem,
\begin{equation}\label{5.95}
\| \frac{2m}{r^{2}} A_{\theta}[P_{\leq N} u] (P_{\leq N} u) + \frac{1}{r^{2}} A_{\theta}[P_{\leq N} u]^{2} (P_{\leq N} u) + A_{0}[P_{\leq N} u](P_{\leq N} u) + g |P_{\leq N} u|^{2} (P_{\leq N} u) \|_{L_{t}^{\infty} L_{x}^{2}} \lesssim N^{2} \| u \|_{L^{2}}^{3}.
\end{equation}
Also, by $(\ref{5.92})$ and the Sobolev embedding theorem,
\begin{equation}\label{5.96}
\aligned
\| P_{\leq \frac{N}{8}} \mathcal N \|_{L_{t}^{2} L_{x}^{\infty}} \lesssim N^{2} \| \mathcal N \|_{L_{t}^{2} L_{x}^{1}} \lesssim N^{2} (\inf_{t \in [a, b]} \| \epsilon(t) \|_{L^{2}}^{2} + \frac{1}{T^{10}})^{1/2}.
\endaligned
\end{equation}
Combining $(\ref{5.95})$ and $(\ref{5.96})$,
\begin{equation}\label{5.97}
\| P_{\leq \frac{N}{8}} u_{t} \|_{L_{t}^{\infty} L_{x}^{2} + L_{t}^{2} L_{x}^{\infty}} \lesssim 1 + (\inf_{t \in [a, b]} \| \epsilon(t) \|_{L^{2}}^{2} + \frac{1}{T^{10}})^{1/2}.
\end{equation}
Now turn to the third terms in $(\ref{5.82})$--$(\ref{5.86})$ with $(\overline{P_{\leq N} u_{t}})$ replaced by $(\overline{P_{\leq \frac{N}{8}} u_{t}})$. First take the third term in $(\ref{5.82})$. By Fourier support arguments,
\begin{equation}\label{5.98}
\aligned
\int (\overline{P_{\leq \frac{N}{8}} u_{t}}) (|P_{\leq N} u|^{2}(P_{\leq N} u) - |u|^{2} u) \lesssim \| P_{> N} u \|_{L_{t}^{2} L_{x}^{\infty}} \| P_{> \frac{N}{8}} u \|_{L_{t}^{2} L_{x}^{\infty} \cap L_{t}^{\infty} L_{x}^{2}} \| u \|_{L_{t}^{\infty} L_{x}^{2}} \| P_{\leq \frac{N}{8}} u_{t} \|_{L_{t}^{2} L_{x}^{\infty} + L_{t}^{\infty} L_{x}^{2}} \\
\lesssim N^{2} (\inf_{t \in [a, b]} \| \epsilon(t) \|_{L^{2}}^{2} + \frac{1}{T^{10}}).
\endaligned
\end{equation}
Next take $(\ref{5.83})$,
\begin{equation}\label{5.99}
A_{\theta}[u] u - A_{\theta}[P_{\leq N} u](P_{\leq N} u) = A_{\theta}[u] (P_{> N} u) + (A_{\theta}[u] - A_{\theta}[P_{\leq N} u])(P_{\leq N} u).
\end{equation}
By standard Fourier support arguments,
\begin{equation}\label{5.100}
\int (\overline{P_{\leq \frac{N}{8}} u_{t}})(P_{> N} u) \frac{2m}{r^{2}} A_{\theta}[u] = \int  (\overline{P_{\leq \frac{N}{8}} u_{t}})(P_{> N} u) P_{> \frac{N}{2}} (\frac{2m}{r^{2}} A_{\theta}[u]) \lesssim N^{2} (\inf_{t \in [a, b]} \| \epsilon(t) \|_{L^{2}}^{2} + \frac{1}{T^{10}}).
\end{equation}
Next,
\begin{equation}\label{5.101}
\aligned
\int \frac{2m}{r^{2}} (P_{\leq N} u)(\overline{P_{\leq \frac{N}{8}} u_{t}}) \{ A_{0}[u] - A_{0}[P_{\leq N} u] \} = \int \frac{2m}{r^{2}} (P_{\leq N} u)(\overline{P_{\leq \frac{N}{8}} u_{t}}) \{ Re \int_{0}^{r} 2 Re((P_{\leq N} u)(P_{> N} u)) s ds \} \\
+ \int \frac{2m}{r^{2}} (P_{\leq N} u)(\overline{P_{\leq \frac{N}{8}} u_{t}}) \{ Re \int_{0}^{r} |P_{> N} u)|^{2} s ds \}.
\endaligned
\end{equation}
Changing the order of integration,
\begin{equation}\label{5.102}
\aligned
 \int \frac{2m}{r^{2}} (P_{\leq N} u)(\overline{P_{\leq \frac{N}{8}} u_{t}}) \{ Re \int_{0}^{r} 2 Re((P_{\leq N} u)(P_{> N} u)) s ds \} \\ = Re \int (P_{> N} u)(\overline{P_{\leq N} u}) \{ Re \int_{r}^{\infty} \frac{2m}{s} (P_{\leq N} u)(\overline{P_{\leq \frac{N}{8}} u_{t}}) ds \}.
 \endaligned
\end{equation}
Following the computations in $(\ref{5.67})$--$(\ref{5.73})$,
\begin{equation}\label{5.103}
(\ref{5.102}) \lesssim \frac{N^{2}}{T^{10}} + N^{2} \inf_{t \in [a, b]} \| \epsilon(t) \|_{L^{2}}^{2}.
\end{equation}
Also,
\begin{equation}\label{5.104}
\| \frac{2m}{r^{2}} \int_{0}^{r} s |P_{> N} u|^{2} ds \|_{L_{t}^{1} L_{x}^{\infty}} \lesssim \| P_{> N} u \|_{L_{t}^{2} L_{x}^{\infty}}^{2}, \qquad \| \frac{2m}{r^{2}} \int_{0}^{r} s |P_{> N} u|^{2} ds \|_{L_{t,x}^{2}} \lesssim \| u \|_{L^{2}} \| P_{> N} u \|_{L_{t}^{2} L_{x}^{\infty}}.
\end{equation}
Combining $(\ref{5.97})$ with $(\ref{5.104})$,
\begin{equation}\label{5.105}
\int \frac{2m}{r^{2}} (\overline{P_{\leq \frac{N}{8}} u_{t}})(P_{\leq N} u) (A_{\theta}[u] - A_{\theta}[P_{\leq N} u) \lesssim N^{2} \inf_{t \in [a, b]} \| \epsilon(t) \|_{L^{2}}^{2} + \frac{N^{2}}{T^{10}}.
\end{equation}

Now turn to $(\ref{5.84})$. Expanding,
\begin{equation}\label{5.106}
A_{\theta}[P_{\leq N} u]^{2} (P_{\leq N} u) - A_{\theta}[u]^{2} u = A_{\theta}[u]^{2} (P_{> N} u) + (A_{\theta}[P_{\leq N} u] - A_{\theta}[u])(A_{\theta}[P_{\leq N} u] + A_{\theta}[u])(P_{\leq N} u).
\end{equation}
Following $(\ref{5.55})$,
\begin{equation}\label{5.107}
\int \frac{1}{r^{2}} (\overline{P_{\leq \frac{N}{8}} u_{t}})(P_{> N} u) A_{\theta}[u]^{2} \lesssim N^{2} \inf_{t \in [a, b]} \| \epsilon(t) \|_{L^{2}}^{2} + \frac{N^{2}}{T^{10}}.
\end{equation}
Next, as in $(\ref{5.101})$, after changing the order of integration and using $(\ref{5.67})$--$(\ref{5.73})$,
\begin{equation}\label{5.108}
\int \frac{1}{r^{2}} (\overline{P_{\leq \frac{N}{8}} u_{t}})(P_{> N} u) (A_{\theta}[P_{\leq N} u] - A_{\theta}[u])(A_{\theta}[P_{\leq N} u] + A_{\theta}[u]) \lesssim N^{2} \inf_{t \in [a, b]} \| \epsilon(t) \|_{L^{2}}^{2} + \frac{N^{2}}{T^{10}}.
\end{equation}
Finally, decompose
\begin{equation}\label{5.109}
A_{0}[P_{\leq N} u] (P_{\leq N} u) - A_{0}[u] u = -A_{0}[u](P_{> N} u) + (A_{0}[P_{\leq N} u] - A_{0}[u]) (P_{\leq N} u).
\end{equation}
Again following $(\ref{5.67})$--$(\ref{5.73})$,
\begin{equation}\label{5.110}
\int (\overline{P_{\leq \frac{N}{8}} u_{t}}) A_{0}[u] (P_{> N} u) = \int P_{> \frac{N}{2}} ((\overline{P_{\leq \frac{N}{8}} u_{t}}) A_{0}[u]) (P_{> N} u) \lesssim N^{2} \inf_{t \in [a, b]} \| \epsilon(t) \|_{L^{2}}^{2} + \frac{N^{2}}{T^{10}}.
\end{equation}
Finally, changing the order of integration,
\begin{equation}\label{5.111}
\aligned
\int (\overline{P_{\leq \frac{N}{8}} u_{t}})(P_{\leq N} u)(A_{0}[P_{\leq N} u] - A_{0}[u]) \\ = \int \{ (\frac{m + A_{\theta}[u]}{r^{2}}) |u|^{2} - (\frac{m + A_{\theta}[P_{\leq N} u]}{r^{2}}) |P_{\leq N} u|^{2} \} \cdot \{ \int_{0}^{r} (\overline{P_{\leq \frac{N}{8}} u_{t}})(P_{\leq N} u) s ds \} r dr dt.
\endaligned
\end{equation}
Then, applying the arguments in $(\ref{5.53})$--$(\ref{5.62})$ to $(\ref{5.111})$ proves that
\begin{equation}\label{5.112}
(\ref{5.111}) \lesssim N^{2} \inf_{t \in [a, b]} \| \epsilon(t) \|_{L^{2}}^{2} + \frac{N^{2}}{T^{10}}.
\end{equation}
This finally proves Proposition $(\ref{p5.6})$.
\end{proof}

\begin{proof}[Proof of Proposition $\ref{p5.4}$]
Now we are finally ready to prove Proposition $\ref{p5.4}$. The proof uses the Morawetz estimate.
\begin{proposition}\label{p5.7}
Let $u$ be a solution to $(\ref{1.1})$ and let $\psi \in C_{0}^{\infty}(\mathbb{R}^{2})$ be a smooth, radially symmetric function such that $\psi(r) = r$ for $r \leq 1$, $\psi(r) = \frac{3}{2}$ for $r > 2$, and $\partial_{r}(\psi(r)) = \phi(r)^{2}$ $\phi(r) \in C_{0}^{\infty}(\mathbb{R}^{2})$, and $\phi(r) \geq 0$. Then if
\begin{equation}\label{5.113}
M(t) = R \int \psi(\frac{r}{R}) Im[\bar{u} \partial_{r} u](t, r) r dr,
\end{equation}
then
\begin{equation}\label{5.114}
\aligned
\frac{d}{dt} M(t) = 2 \int \phi^{2}(\frac{r}{R}) r |u_{r}|^{2} + 2 \int \phi^{2}(\frac{r}{R}) (\frac{m + A_{\theta}[u]}{r})^{2} |u|^{2} - g \int \phi^{2}(\frac{r}{R}) |u|^{4} r \\  + O(\int_{r \geq R} \frac{1}{r^{2}} |u|^{2}) + O(\int_{r \geq R} |u|^{4}).
\endaligned
\end{equation}
\end{proposition}
\begin{proof}
This follows by direct computation and integrating by parts.
\end{proof}

Taking $R = T^{1/25}$, since $Q$ is rapidly decreasing,
\begin{equation}\label{5.115}
\int_{r \geq R} |u(t, x)|^{2} \frac{1}{r^{2}} dx \lesssim \frac{1}{R \lambda(t)^{2}} \frac{1}{T^{10}} + \frac{1}{R \lambda(t)^{2}} \| \epsilon(t) \|_{L^{2}}^{2}.
\end{equation}
Therefore,
\begin{equation}\label{5.116}
\int_{a}^{b} \int_{r \geq R} |u(t, x)|^{2} \frac{1}{r^{2}} dx dt \ll \frac{1}{T^{9}} + \int_{a}^{b} \frac{1}{\lambda(t)^{2}} \| \epsilon(t) \|_{L^{2}}^{2} dt.
\end{equation}
Also using the fact that $Q$ is rapidly decreasing combined with the Strichartz estimates, for any $m \geq 0$,
\begin{equation}\label{5.117}
\int_{m}^{m + 1} \| \epsilon(s) \|_{L^{4}}^{4} ds \lesssim (\int_{m}^{m + 1} \| \epsilon(s) \|_{L^{2}}^{2})^{2},
\end{equation}
\begin{equation}\label{5.118}
\int_{a}^{b} \int_{r \geq R} |u(t, x)|^{4} dx dt \lesssim \int_{a}^{b} \frac{1}{R \lambda(t)^{2}} \frac{1}{T^{10}} dt + \int_{a}^{b} \frac{1}{\lambda(t)^{4}} \int |\epsilon(t, \frac{x}{\lambda(t)})|^{4} dx dt \leq \frac{1}{R} \frac{1}{T^{9}} + \eta_{\ast}^{2} \int_{a}^{b} \frac{1}{\lambda(t)^{2}} \| \epsilon(t) \|_{L^{2}}^{2} dt.
\end{equation}
Replacing $u$ by $P_{\leq N} u$ with $N = T^{1/3}$,
\begin{equation}\label{5.119}
\aligned
2 \int \phi^{2}(\frac{r}{R}) r |\partial_{r} P_{\leq N} u|^{2} + 2 \int \phi^{2}(\frac{r}{R}) (\frac{m + A_{\theta}[P_{\leq N} u]}{r})^{2} |P_{\leq N} u|^{2} - g \int \phi^{2}(\frac{r}{R}) |P_{\leq N} u|^{4} r \\
= 4 E[\phi(\frac{r}{R}) P_{\leq N} u] + O(\int_{r \geq R} \frac{1}{r^{2}} |u|^{2}) + O(\int_{r \geq R} |u|^{4}).
\endaligned
\end{equation}
Therefore,
\begin{equation}\label{5.120}
\int_{a}^{b} \| \epsilon(t) \|_{L^{2}}^{2} \frac{1}{\lambda(t)^{2}} dt \lesssim R \int \psi(\frac{r}{R}) Im[\overline{P_{\leq N} u} \partial_{r} P_{\leq N} u]|_{a}^{b} + \mathcal E,
\end{equation}
where $\mathcal E$ are the error terms arising from frequency truncation, see $(\ref{5.82})$--$(\ref{5.86})$. Now then, using $(\ref{5.90})$ and the fact that $R \psi(\frac{r}{R})$ is smooth,
\begin{equation}\label{5.121}
\aligned
 \int Re[P_{> N} (-g |u|^{2} u + \frac{2m}{r^{2}} A_{\theta}[u] u + \frac{A_{\theta}[u]^{2}}{r^{2}} u + A_{0}[u] u) \cdot \partial_{r} \overline{P_{\leq N} u}] R \psi(\frac{r}{R}) \\
 \lesssim RN \| P_{> N} (-g |u|^{2} u + \frac{2m}{r^{2}} A_{\theta}[u] u + \frac{A_{\theta}[u]^{2}}{r^{2}} u + A_{0}[u] u) \|_{L_{t}^{2} L_{x}^{1}} \| P_{\frac{N}{2} < \cdot < N} u \|_{L_{t}^{2} L_{x}^{\infty}} \\
 + RN \| P_{> N} (-g |u|^{2} u + \frac{2m}{r^{2}} A_{\theta}[u] u + \frac{A_{\theta}[u]^{2}}{r^{2}} u + A_{0}[u] u) \|_{L_{t}^{2} L_{x}^{1}} \| P_{\leq \frac{N}{2}} u \|_{L_{t}^{2} L_{x}^{\infty}} \| P_{> \frac{N}{2}} (\psi(\frac{r}{R})) \|_{L^{\infty}} \\
 \lesssim RN \inf_{t \in [a, b]} \| \epsilon(t) \|_{L^{2}}^{2} + \frac{RN}{T^{10}}.
 \endaligned
\end{equation}
Integrating by parts, the same estimate holds for
\begin{equation}\label{5.122}
 \int Re[\overline{P_{\leq N} u} \cdot \partial_{r} (P_{> N} (-g |u|^{2} u + \frac{2m}{r^{2}} A_{\theta}[u] u + \frac{A_{\theta}[u]^{2}}{r^{2}} u + A_{0}[u] u) ))] R \psi(\frac{r}{R}).
\end{equation}
Meanwhile, recalling $(\ref{5.91})$, and using $(\ref{5.92})$--$(\ref{5.112})$ along with the fact that $\psi(\frac{r}{R})$ is smooth,
\begin{equation}\label{5.123}
R \int Re[\bar{\mathcal N} \partial_{r} P_{\leq N} u] \psi(\frac{r}{R}) dx dt + R \int Re[(\overline{P_{\leq N} u}) \partial_{r} \mathcal N] \psi(\frac{r}{R}) \lesssim RN \inf_{t \in [a, b]} \| \epsilon(t) \|_{L^{2}}^{2} + \frac{RN}{T^{10}}.
\end{equation}
Therefore, $\int_{a}^{b} \mathcal E dt \lesssim RN \inf_{t \in [a, b]} \| \epsilon(t) \|_{L^{2}}^{2} + \frac{RN}{T^{10}}$. Since by Proposition $\ref{p5.6}$,
\begin{equation}\label{5.124}
R \int \psi(\frac{r}{R}) Im[\bar{P_{\leq N} \epsilon} \partial_{r} P_{\leq N} \epsilon] \lesssim RN \| \epsilon \|_{L^{2}}^{2} \lesssim T^{1/50} RN E[u(t)] \lesssim \frac{T^{1/50} RN}{T} \int_{a}^{b} \| \epsilon(t) \|_{L^{2}}^{2} \lambda(t)^{-2} dt + \frac{1}{T^{9}},
\end{equation}
and therefore $(\ref{5.46})$ holds.

\end{proof}

\section{An $L_{s}^{p}$ bound on $\| \epsilon(s) \|_{L^{2}}$ when $p > 1$}
As in the case of the two-dimensional mass-critical problem, Proposition $\ref{p5.4}$ implies that $\| \epsilon(s) \|_{L^{2}}$ lies in $L_{s}^{p}$ for any $p > 1$.
\begin{proposition}\label{p6.1}
Let $u$ be a solution to $(\ref{1.1})$ that satisfies $\| u \|_{L^{2}} = \| Q \|_{L^{2}}$, and suppose
\begin{equation}\label{6.1}
\sup_{s \in [0, \infty)} \| \epsilon(s) \|_{L^{2}} \leq \eta_{\ast},
\end{equation}
and $\| \epsilon(0) \|_{L^{2}} = \eta_{\ast}$. Then
\begin{equation}\label{6.2}
\int_{0}^{\infty} \| \epsilon(s) \|_{L^{2}}^{2} ds \lesssim \eta_{\ast},
\end{equation}
with implicit constant independent of $\eta_{\ast}$ when $\eta_{\ast} \ll 1$ is sufficiently small.

Furthermore, for any $j \in \mathbb{Z}_{\geq 0}$, let
\begin{equation}\label{6.3}
s_{j} = \inf \{ s \in [0, \infty) : \| \epsilon(s) \|_{L^{2}} = 2^{-j} \eta_{\ast} \}.
\end{equation}
By definition, $s_{0} = 0$, and the continuity of $\| \epsilon(s) \|_{L^{2}}$ combined with sequential convergence of blowup solutions implies that such an $s_{j}$ exists for any $j > 0$. Then,
\begin{equation}\label{6.4}
\int_{s_{j}}^{\infty} \| \epsilon(s) \|_{L^{2}}^{2} ds \lesssim 2^{-j} \eta_{\ast},
\end{equation}
for each $j \geq 0$, with implicit constant independent of $\eta_{\ast}$.
\end{proposition}

\begin{proof}
Set $T_{\ast} = \frac{1}{\eta_{\ast}}$ and suppose that $T_{\ast}$ is sufficiently large such that Proposition $\ref{p6.1}$ holds. Then by $(\ref{6.1})$, for any $s' \geq 0$,
\begin{equation}\label{6.5}
|\sup_{s \in [s', s' + T_{\ast}]} \ln(\lambda(s)) - \inf_{s \in [s', s' + T_{\ast}]} \ln(\lambda(s))| \lesssim 1,
\end{equation}
with implicit constant independent of $s' \geq 0$. Let $J$ be the largest dyadic integer that satisfies
\begin{equation}\label{6.6}
J = 2^{j_{\ast}} \leq -\ln(\eta_{\ast})^{1/4}.
\end{equation}
By $(\ref{6.5})$ and the triangle inequality,
\begin{equation}\label{6.7}
\aligned
|\sup_{s \in [s', s' + J T_{\ast}]} \ln(\lambda(s)) - \inf_{s \in [s', s' + J T_{\ast}]} \ln(\lambda(s))| \lesssim J,
\endaligned
\end{equation}
and therefore,
\begin{equation}\label{6.8}
\frac{\sup_{s \in [s', s' + 3 J T_{\ast}]} \lambda(s)}{\inf_{s \in [s', s' + 3JT^{\ast}]} \lambda(s)} \lesssim T_{\ast}^{\frac{1}{500}}.
\end{equation}
Rescale so that
\begin{equation}\label{6.9}
1 \leq \lambda(s) \leq  T_{\ast}^{\frac{1}{500}}, \qquad \text{for any} \qquad s \in [s', s' + 3J T_{\ast}].
\end{equation}

Utilizing Proposition $\ref{p5.4}$ on $[s', s' + J T_{\ast}]$, for any $s' \geq 0$,
\begin{equation}\label{6.10}
\int_{s'}^{s' + J T_{\ast}} \| \epsilon(s) \|_{L^{2}}^{2} ds \lesssim \| \epsilon(s') \|_{L^{2}} + \| \epsilon(s' + J T_{\ast}) \|_{L^{2}} + O(\frac{1}{J^{8} T_{\ast}^{8}}).
\end{equation}
Note that the left hand side of $(\ref{6.10})$ is scale invariant.

Moreover, for any $s' > J T_{\ast}$,
\begin{equation}\label{6.11}
\int_{s'}^{s' + J T_{\ast}} \| \epsilon(s) \|_{L^{2}}^{2} ds \lesssim \inf_{s \in [s' - J T_{\ast}, s']} \| \epsilon(s) \|_{L^{2}} + \inf_{s \in [s' + J T_{\ast}, s' + 2J T_{\ast}]} \| \epsilon(s) \|_{L^{2}} + O(\frac{1}{J^{8} T_{\ast}^{8}}).
\end{equation}
In particular, for a fixed $s' \geq 0$,
\begin{equation}\label{6.12}
\sup_{a > 0} \int_{s' + a J T_{\ast}}^{s' + (a + 1) J T_{\ast}} \| \epsilon(s) \|_{L^{2}}^{2} \lesssim \frac{1}{J^{1/2} T_{\ast}^{1/2}} (\sup_{a \geq 0} \int_{s' + a J T_{\ast}}^{s' + (a + 1) J T_{\ast}} \| \epsilon(s) \|_{L^{2}}^{2} ds)^{1/2} + O(\frac{1}{J^{8} T_{\ast}^{8}}).
\end{equation}
Meanwhile, when $a = 0$,
\begin{equation}\label{6.13}
 \int_{s'}^{s' + J T_{\ast}} \| \epsilon(s) \|_{L^{2}}^{2} \lesssim \| \epsilon(s') \|_{L^{2}} + \frac{1}{J^{1/2} T_{\ast}^{1/2}} (\sup_{a \geq 0} \int_{s' + a J T_{\ast}}^{s' + (a + 1) J T_{\ast}} \| \epsilon(s) \|_{L^{2}}^{2} ds)^{1/2} + O(\frac{1}{J^{8} T_{\ast}^{8}}).
\end{equation}
Therefore, taking $s' = s_{j_{\ast}}$,
\begin{equation}\label{6.14}
\sup_{a \geq 0} \int_{s_{j_{\ast}} + a J T_{\ast}}^{s_{j_{\ast}} + (a + 1) J T_{\ast}} \| \epsilon(s) \|_{L^{2}}^{2} ds \lesssim 2^{-j_{\ast}} \eta_{\ast} + O(2^{-8 j_{\ast}} \eta_{\ast}^{8}).
\end{equation}
Then by the triangle inequality,
\begin{equation}\label{6.15}
\sup_{s' \geq s_{j_{\ast}}} \int_{s'}^{s' + J T_{\ast}} \| \epsilon(s) \|_{L^{2}}^{2} ds \lesssim 2^{-j_{\ast}} \eta_{\ast},
\end{equation}
and by H{\"o}lder's inequality,
\begin{equation}\label{6.16}
\sup_{s' \geq s_{j_{\ast}}} \int_{s'}^{s' + J T_{\ast}} \| \epsilon(s) \|_{L^{2}} ds \lesssim 1.
\end{equation}

Repeating this argument, Proposition $\ref{p6.1}$ can be proved by induction. Indeed, fix a constant $C < \infty$ and suppose that there exists a positive integer $n_{0}$ such that for all integers $0 \leq n \leq n_{0}$,
\begin{equation}\label{6.17}
\sup_{s' \geq s_{nj_{\ast}}} \int_{s'}^{s' + J^{n} T_{\ast}} \| \epsilon(s) \|_{L^{2}} ds \leq C, \qquad \sup_{s' \geq s_{nj_{\ast}}} \int_{s'}^{s' + J^{n} T_{\ast}} \| \epsilon(s) \|_{L^{2}}^{2} ds \leq C J^{-n} \eta_{\ast}.
\end{equation}
Then for $s' \geq s_{n j_{\ast}}$,
\begin{equation}\label{6.18}
\frac{\sup_{s \in [s', s' + 3 J^{n + 1} T_{\ast}]} \lambda(s)}{\inf_{s \in [s', s' + 3 J^{n + 1} T_{\ast}]} \lambda(s)} \lesssim T_{\ast}^{\frac{1}{500}}.
\end{equation}

Then by Proposition $\ref{p5.4}$,
\begin{equation}\label{6.19}
\sup_{s' \geq s_{(n + 1) j_{\ast}}} \int_{s'}^{s' + J^{n + 1} T_{\ast}} \| \epsilon(s) \|_{L^{2}}^{2} ds \leq C J^{-(n + 1)} \eta_{\ast},
\end{equation}
and by H{\"o}lder's inequality,
\begin{equation}\label{6.20}
\sup_{s' \geq s_{(n + 1) j_{\ast}}} \int_{s'}^{s' + J^{n + 1} T_{\ast}} \| \epsilon(s) \|_{L^{2}} ds \leq C.
\end{equation}
Therefore, $(\ref{6.17})$ holds for any integer $n > 0$.\medskip

Now take any $j \in \mathbb{Z}$ and suppose $n j_{\ast} < j \leq (n + 1) j_{\ast}$. Then $(\ref{6.19})$ holds on $[s_{j} + a J^{n + 1} T_{\ast}, s_{j} + (a + 1) J^{n + 1} T_{\ast}]$ for any $a \geq 0$, so
by Proposition $\ref{p5.4}$,
\begin{equation}\label{6.21}
\sup_{a \geq 0} \int_{s_{j} + a J^{n + 1} T_{\ast}}^{s_{j} + (a + 1) J^{n + 1} T_{\ast}} \| \epsilon(s) \|_{L^{2}}^{2} ds \lesssim 2^{-j} \eta_{\ast},
\end{equation}
and therefore by H{\"o}lder's inequality, for any $s' \geq s_{j}$,
\begin{equation}\label{6.22}
\sup_{s' \geq s_{j}} \int_{s'}^{s' + 2^{j} T_{\ast}} \| \epsilon(s) \|_{L^{2}} ds \lesssim 1,
\end{equation}
with bound independent of $j$. Inequalities $(\ref{6.21})$ and $(\ref{6.22})$ imply that the conditions of Proposition $\ref{p5.4}$ hold on $[s', s' + 3 \cdot 2^{j} J T_{\ast}]$ for any $s' \geq s_{j}$, so
\begin{equation}\label{6.23}
\int_{s_{j}}^{s_{j} + 2^{j} J T_{\ast}} \| \epsilon(s) \|_{L^{2}}^{2} \lesssim 2^{-j} \eta_{\ast},
\end{equation}
and therefore, by the mean value theorem,
\begin{equation}\label{6.24}
\inf_{s \in [s_{j}, s_{j} + 2^{j} J T_{\ast}]} \| \epsilon(s) \|_{L^{2}} \lesssim 2^{-j} \eta_{\ast} J^{-1/2},
\end{equation}
which implies
\begin{equation}\label{6.25}
s_{j + 1} \in [s_{j}, s_{j} + 2^{j} J T_{\ast}].
\end{equation}
Therefore, by $(\ref{6.23})$ and H{\"o}lder's inequality,
\begin{equation}\label{6.26}
\int_{s_{j}}^{s_{j + 1}} \| \epsilon(s) \|_{L^{2}}^{2} ds \lesssim 2^{-j} \eta_{\ast}, \qquad \text{and} \qquad \int_{s_{j}}^{s_{j + 1}} \| \epsilon(s) \|_{L^{2}} ds \lesssim 1,
\end{equation}
with constant independent of $j$. Summing in $j$ gives $(\ref{6.2})$ and $(\ref{6.4})$.
\end{proof}

Now then, for any $1 < p < \infty$, $(\ref{6.26})$ implies
\begin{equation}\label{6.27}
(\int_{s_{j}}^{s_{j + 1}} \| \epsilon(s) \|_{L^{2}}^{p} ds) \lesssim \eta_{\ast}^{p - 1} 2^{-j(p - 1)}, 
\end{equation}
which implies that $\| \epsilon(s) \|_{L^{2}}$ belongs to $L_{s}^{p}$ for any $p > 1$, but not $L_{s}^{1}$.

Comparing $(\ref{6.27})$ to the pseudoconformal transformation of the soliton, for $0 < t < 1$,
\begin{equation}\label{6.28}
\lambda(t) \sim t, \qquad \text{and} \qquad \| \epsilon(t) \|_{L^{2}} \sim t,
\end{equation}
so
\begin{equation}\label{6.29}
\int_{0}^{1} \| \epsilon(t) \|_{L^{2}} \lambda(t)^{-2} dt = \infty,
\end{equation}
but for any $p > 1$,
\begin{equation}\label{6.30}
\int_{0}^{1} \| \epsilon(t) \|_{L^{2}}^{p} \lambda(t)^{-2} dt < \infty.
\end{equation}
For the soliton, $\epsilon(s) \equiv 0$ for any $s \in \mathbb{R}$, so obviously, $\| \epsilon(s) \|_{L^{2}} \in L_{s}^{p}$ for $1 \leq p \leq \infty$.

\section{Monotonicity of $\lambda$}
Now prove monotonicity of $\lambda$, as in the mass-critical problem.

\begin{proposition}\label{p7.1}
For any $s \geq 0$, let
\begin{equation}\label{7.1}
\tilde{\lambda}(s) = \inf_{\tau \in [0, s]} \lambda(\tau).
\end{equation}
Then for any $s \geq 0$,
\begin{equation}\label{7.2}
1 \leq \frac{\lambda(s)}{\tilde{\lambda}(s)} \leq 3.
\end{equation}
\end{proposition}

\begin{proof}
Suppose there exist $0 \leq s_{-} \leq s_{+} < \infty$ satisfying
\begin{equation}\label{7.3}
\frac{\lambda(s_{+})}{\lambda(s_{-})} = e.
\end{equation}
Then $u$ is a soliton solution, which contradicts $(\ref{7.3})$. Recall that
\begin{equation}\label{7.4}
\epsilon(t, x) = e^{i \gamma(t)} \lambda(t) u(t, \lambda(t) x) - Q(x).
\end{equation}
Taking the derivative of $(\ref{7.4})$ in time and plugging in
\begin{equation}\label{7.5}
\partial_{t} u = i \Delta u + i g |u|^{2} u - i A_{0}[u] u - \frac{2mi}{r^{2}} A_{\theta}[u] u - i \frac{A_{\theta}[u]^{2}}{r^{2}} u,
\end{equation}
and using the formula $\Lambda = 1 + x \cdot \nabla$,
\begin{equation}\label{7.6}
\aligned
\epsilon_{s} = i \gamma_{s} [\epsilon + Q] + \frac{\lambda_{s}}{\lambda} \Lambda [\epsilon + Q] + i \Delta [\epsilon + Q] + gi |\epsilon + Q|^{2} (\epsilon + Q) - i A_{0}[\epsilon + Q](\epsilon + Q) \\ - \frac{2mi}{r^{2}} A_{\theta}[\epsilon + Q](\epsilon + Q) - \frac{i}{r^{2}} A_{\theta}[\epsilon + Q]^{2} (\epsilon + Q).
\endaligned
\end{equation}
Now then,
\begin{equation}\label{7.7}
\Delta Q + g |Q|^{2} Q - i A_{0}[Q] Q - \frac{2m}{r^{2}} A_{\theta}[Q] Q - \frac{1}{r^{2}} A_{\theta}[Q] Q = \alpha Q.
\end{equation}
Plugging $(\ref{7.7})$ into $(\ref{7.6})$,
\begin{equation}\label{7.8}
\aligned
\epsilon_{s} = i \gamma_{s} \epsilon + i (\gamma_{s} + \alpha) Q + \frac{\lambda_{s}}{\lambda} \Lambda Q + \frac{\lambda_{s}}{\lambda} \Lambda \epsilon + i \Delta \epsilon + gi \{ |\epsilon + Q|^{2} (\epsilon + Q) - |Q|^{2} Q \} \\ - i \{ A_{0}[\epsilon + Q](\epsilon + Q) - A_{0}[Q] Q \} - \frac{2mi}{r^{2}} \{ A_{\theta}[\epsilon + Q](\epsilon + Q) - A_{\theta}[Q] Q \} \\ - \frac{i}{r^{2}} \{ A_{\theta}[\epsilon + Q]^{2} (\epsilon + Q) - A_{\theta}[Q] Q \}.
\endaligned
\end{equation}
Now decompose $\epsilon$ into its real and imaginary parts, $\epsilon = \epsilon_{1} + i \epsilon_{2}$. Taking the real parts of both sides of $(\ref{7.8})$,
\begin{equation}\label{7.9}
\aligned
\partial_{s} \epsilon_{1} = -\gamma_{s} \epsilon_{2} + \frac{\lambda_{s}}{\lambda} \Lambda Q + \frac{\lambda_{s}}{\lambda} \Lambda \epsilon_{1} - \Delta \epsilon_{2} - g Q^{2} \epsilon_{2} + O(Q \epsilon^{2} + \epsilon^{3}) \\  + \frac{1}{r^{2}} A_{\theta}[Q] \epsilon_{2} + \frac{2m}{r^{2}} A_{\theta}[Q] \epsilon_{2} + A_{0}[Q] \epsilon_{2} \\ + O_{m}(\frac{1}{r^{2}} \int_{0}^{r} \{ |\epsilon| Q + |\epsilon|^{2} \} s ds \cdot \epsilon_{2}) + O(\| \epsilon \|_{L^{2}} |\epsilon|).
\endaligned
\end{equation} 

Now compute the virial identity from \cite{merle2005blow},
\begin{equation}\label{7.10}
\aligned
\frac{d}{ds} (\epsilon, |x|^{2} Q) = -\gamma_{s} (\epsilon_{2}, |x|^{2} Q) + \frac{\lambda_{s}}{\lambda} (\Lambda Q, |x|^{2} Q) - (\Delta \epsilon_{2} + g Q^{2} \epsilon_{2} - \frac{1}{r^{2}} A_{\theta}[Q] \epsilon_{2} - \frac{2m}{r^{2}} A_{\theta}[Q] \epsilon_{2} - A_{0} [Q] \epsilon_{2}, |x|^{2} Q) \\ + O(\| \epsilon \|_{L^{2}}^{2} + \| \epsilon \|_{L^{\infty}}^{2}).
\endaligned
\end{equation}
Since
\begin{equation}\label{7.11}
\Delta Q + g Q^{3} - \frac{1}{r^{2}} A_{\theta}[Q] - \frac{2m}{r^{2}} A_{\theta}[Q] - A_{0} [Q] Q = \alpha Q,
\end{equation}
then integrating by parts,
\begin{equation}\label{7.12}
(\Delta \epsilon_{2} + g Q^{2} \epsilon_{2} - \frac{1}{r^{2}} A_{\theta}[Q] \epsilon_{2} - \frac{2m}{r^{2}} A_{\theta}[Q] \epsilon_{2} - A_{0} [Q] \epsilon_{2}, |x|^{2} Q) = (4 \epsilon_{2}, \Lambda Q) + \alpha (\epsilon_{2}, Q).
\end{equation}
Therefore,
\begin{equation}\label{7.13}
\frac{d}{ds} (\epsilon, |x|^{2} Q) = -(\gamma_{s} + \alpha) (\epsilon_{2}, |x|^{2} Q) + \frac{\lambda_{s}}{\lambda} (\Lambda Q, |x|^{2} Q) - 4 (\epsilon_{2}, \Lambda Q) + O(\| \epsilon \|_{L^{2}}^{2} + \| \epsilon \|_{L^{\infty}}^{2}).
\end{equation}

Using Proposition $\ref{p6.1}$, $(\ref{7.13})$, the fundamental theorem of calculus, and the fact that $(|x|^{2} Q, \Lambda Q) = -\| xQ \|_{L^{2}}^{2}$,
\begin{equation}\label{7.14}
\| x Q \|_{L^{2}}^{2} + 4 \int_{s_{-}}^{s_{+}} (\epsilon_{2}, Q + x \cdot \nabla Q)_{L^{2}} = O(\eta_{\ast}).
\end{equation}
Therefore, there exists $s' \in [s_{-}, s_{+}]$ such that
\begin{equation}\label{7.15}
(\epsilon_{2}, Q + x \cdot \nabla Q)_{L^{2}} < 0.
\end{equation}

Since $s' \geq 0$, there exists some $j \geq 0$ such that $s_{j} \leq s' + T_{\ast} < s_{j + 1}$. Using the proof of Proposition $\ref{p6.1}$,
\begin{equation}\label{7.16}
\int_{s'}^{s_{j + 1 + J}} |\frac{\lambda_{s}}{\lambda}| ds \lesssim J.
\end{equation}
Then by Proposition $\ref{p5.4}$, $(\ref{7.16})$ implies
\begin{equation}\label{7.17}
\int_{s'}^{s_{j + 1 + J}} \| \epsilon(s) \|_{L^{2}}^{2} ds \lesssim 2^{-(j + 1 + J)} \eta_{\ast},
\end{equation}
and therefore by definition of $s_{j + 1 + J}$,
\begin{equation}\label{7.18}
\int_{s'}^{s_{j + 1 + J}} \| \epsilon(s) \|_{L^{2}} ds \lesssim 1.
\end{equation}

Arguing by induction, suppose that for some $1 \leq k \leq k_{0}$,
\begin{equation}\label{7.19}
\int_{s'}^{s_{j + k}} \| \epsilon(s) \|_{L^{2}}^{2} ds \lesssim 2^{-j - k} \eta_{\ast},
\end{equation}
and
\begin{equation}\label{7.20}
\int_{s'}^{s_{j + k}} \| \epsilon(s) \|_{L^{2}} ds \lesssim 1,
\end{equation}
with implicit constant independent of $k$. By Proposition $\ref{p6.1}$,
\begin{equation}\label{7.21}
\int_{s'}^{s_{j + k + J}} \| \epsilon(s) \|_{L^{2}}^{2} ds \lesssim 2^{-j - k} \eta_{\ast},
\end{equation}
and
\begin{equation}\label{7.22}
\int_{s'}^{s_{j + k + J}} \| \epsilon(s) \|_{L^{2}} ds \lesssim J.
\end{equation}
Then by Proposition $\ref{p5.4}$,
\begin{equation}\label{7.23}
\int_{s'}^{s_{j + k + J}} \| \epsilon(s) \|_{L^{2}}^{2} ds \lesssim 2^{-j - k - J} \eta_{\ast},
\end{equation}
and
\begin{equation}\label{7.24}
\int_{s'}^{s_{j + k + J}} \| \epsilon(s) \|_{L^{2}} ds \lesssim 1,
\end{equation}
for $1 \leq k \leq k_{0} + J$. Therefore, $(\ref{7.23})$ and $(\ref{7.24})$ hold for any $k$, with implicit constant independent of $k$.

Taking $k \rightarrow \infty$,
\begin{equation}\label{7.25}
\int_{s'}^{\infty} \| \epsilon(s) \|_{L^{2}}^{2} ds = 0,
\end{equation}
which implies that $\epsilon(s) = 0$ for all $s \geq s'$. Therefore, $u$ is a soliton solution.
\end{proof}

\section{Rigidity}
To prove Theorem $\ref{t1.2}$, we prove that if $u$ is global solution, $u$ is a soliton, but if $u$ is a finite time blowup solution, then $u$ is a pseudoconformal transformation of a soliton.

\begin{theorem}\label{t8.1}
If $u$ is a solution to $(\ref{1.1})$ that satisfies $dist(u, \mathcal M) \leq \eta_{\ast}$ for all $t \geq 0$, and furthermore, if
\begin{equation}\label{8.1}
\sup(I) = \infty,
\end{equation}
then $u$ is equal to a soliton solution.
\end{theorem}
\begin{proof}
For any integer $k \geq 0$, let
\begin{equation}\label{8.2}
I(k) = \{ s \geq 0 : 2^{-k + 2} \leq \tilde{\lambda}(s) \leq 2^{-k + 3} \}.
\end{equation}
Then by Proposition $\ref{p7.1}$,
\begin{equation}\label{8.3}
2^{-k} \leq \lambda(s) \leq 2^{-k + 3},
\end{equation}
for all $s \in I(k)$. The fact that $\sup(I) = \infty$ implies that
\begin{equation}\label{8.4}
\sum 2^{-2k} |I(k)| = \infty.
\end{equation}
If $\lambda(s) \rightarrow 0$ as $s \rightarrow \infty$, then there exists a sequence $k_{n} \nearrow \infty$ such that
\begin{equation}\label{8.5}
|I(k_{n})| 2^{-2k_{n}} \geq \frac{1}{k_{n}^{2}}, \qquad \text{and} \qquad I(k) \leq 2^{2k_{n}} k_{n}^{-2}, \qquad \forall k \leq k_{n}.
\end{equation}
If $\inf_{s \geq 0} \lambda(s) > 0$, then there exists some $s_{0}$ such that
\begin{equation}\label{8.6}
\frac{\sup_{s \geq s_{0}} \lambda(s)}{\inf_{s \geq s_{0}} \lambda(s)} \leq 2.
\end{equation}

Now let $I(k_{n}) = [a_{n}, b_{n}]$. In the case of $(\ref{8.6})$, let $a_{n} = s_{0}$ and $b_{n} = s_{0} + 2^{4 k_{n}}$. By $(\ref{7.13})$,
\begin{equation}\label{8.7}
|\int_{a_{n}}^{a_{n} + \frac{b_{n} - a_{n}}{4}} (\epsilon_{2}, Q + x \cdot \nabla Q) ds \lesssim 1|, \qquad |\int_{b_{n} - \frac{b_{n} - a_{n}}{4}}^{b_{n}} (\epsilon_{2}, Q + x \cdot \nabla Q) ds \lesssim 1|.
\end{equation}
Therefore, there exists $s_{-} \in [a_{n}, a_{n} + \frac{b_{n} - a_{n}}{4}]$, $s_{+} \in [b_{n} - \frac{b_{n} - a_{n}}{4}, b_{n}]$ such that
\begin{equation}\label{8.8}
|(\epsilon_{2}, Q + x \cdot \nabla Q)(s_{-})|, \qquad |(\epsilon_{2}, Q + x \cdot \nabla Q)(s_{+})| \lesssim 2^{-2k_{n}} k_{n}^{2}.
\end{equation}

Plugging $(\ref{8.8})$ into $(\ref{5.46})$,
\begin{equation}\label{8.9}
\int_{s_{-}}^{s_{+}} \| \epsilon(s) \|_{L^{2}}^{2} ds \lesssim 2^{-2k_{n}} k_{n}^{2}.
\end{equation}
Again by the intermediate value theorem,
\begin{equation}\label{8.10}
\inf_{s \in [s_{-}, s_{+}]} \| \epsilon(s) \|_{L^{2}}^{2} \lesssim 2^{-4k_{n}} k_{n}^{4}.
\end{equation}
After rescaling $\lambda(s_{-}) = 1$, plugging $(\ref{8.10})$ into Proposition $\ref{p5.6}$, and then rescaling back,
\begin{equation}\label{8.11}
E(P_{\leq \frac{4}{3} k_{n}} u)(0) \lesssim 2^{\frac{8 k_{n}}{3}} 2^{-4 k_{n}} k_{n}^{4} \rightarrow 0, \qquad \text{as} \qquad n \rightarrow \infty.
\end{equation}
Therefore, $E(u) = 0$.
\end{proof}

Now turn to a finite time blowup solution. Suppose without loss of generality that $\sup(I) = 0$, and
\begin{equation}\label{8.12}
\sup_{-1 < t < 0} \| \epsilon(t) \|_{L^{2}} \leq \eta_{\ast}.
\end{equation}
Then decomposing $u$,
\begin{equation}\label{8.13}
u(t,x) = \frac{e^{-i \gamma(t)}}{\lambda(t)} Q(\frac{x}{\lambda(t)}) + \frac{e^{-i \gamma(t)} }{\lambda(t)} \epsilon(t,\frac{x }{\lambda(t)}).
\end{equation}
Then apply the pseudoconformal transformation to $u(t,x)$. For $-\infty < t < -1$, let
\begin{equation}\label{8.14}
\aligned
v(t,x) = \frac{1}{t} \overline{u(\frac{1}{t}, \frac{x}{t})} e^{i |x|^{2}/4t} = \frac{1}{t} \frac{e^{i \gamma(1/t)} }{\lambda(1/t)} Q(\frac{x }{t \lambda(1/t)}) e^{i |x|^{2}/4t} + \frac{1}{t} \frac{e^{i \gamma(1/t)}}{\lambda(1/t)} \overline{\epsilon(\frac{1}{t}, \frac{x }{t \lambda(1/t)})} e^{i |x|^{2}/4t}.
\endaligned
\end{equation}
Since the $L^{2}$ norm is preserved by the pseudoconformal transformation,
\begin{equation}\label{8.15}
\aligned
\lim_{t \searrow -\infty} \| \frac{1}{t} \frac{e^{i \gamma(1/t)} }{\lambda(1/t)} \overline{\epsilon(\frac{1}{t}, \frac{x }{t \lambda(1/t)})} e^{i |x|^{2}/4t} \|_{L^{2}} = 0, \qquad \text{and} \\ 
\qquad \sup_{-\infty < t < -1}  \| \frac{1}{t} \frac{e^{i \gamma(1/t)} }{\lambda(1/t)} \overline{\epsilon(\frac{1}{t}, \frac{x}{t \lambda(1/t)})} e^{i x^{2}/4t} \|_{L^{2}} \leq \eta_{\ast}.
\endaligned
\end{equation}

Since
\begin{equation}\label{8.16}
 \frac{1}{t} \frac{e^{i \gamma(1/t)} }{\lambda(1/t)} Q(\frac{x}{t \lambda(1/t)})
\end{equation}
is in the form of $\frac{e^{i \tilde{\gamma}(t)} }{\tilde{\lambda}(t)} Q(\frac{x}{\tilde{\lambda}(t)})$, it only remains to estimate
\begin{equation}\label{8.17}
 \| \frac{1}{t} \frac{e^{i \gamma(1/t)} }{\lambda(1/t)} Q(\frac{x }{t \lambda(1/t)}) (e^{i |x|^{2}/4t} - 1) \|_{L^{2}}.
\end{equation}

For any $k \geq 0$, $\lambda(s) \sim 2^{-k}$ for all $s \in I(k)$. Furthermore, $\| \epsilon(t) \|_{L^{2}} \rightarrow 0$ as $t \nearrow 0$ implies that there exists a sequence $c_{k} \nearrow \infty$ such that
\begin{equation}\label{8.18}
|I(k)| \geq c_{k}, \qquad \text{for all} \qquad k \geq 0.
\end{equation}
Therefore, there exists $r(t) \searrow 0$ as $t \nearrow 0$ such that
\begin{equation}\label{14.25}
\lambda(t) \leq t^{1/2} r(t), \qquad \text{so} \qquad \lambda(1/t) \leq t^{-1/2} r(1/t).
\end{equation}
Therefore, since $Q$ is rapidly decreasing,
\begin{equation}\label{8.20}
\lim_{t \searrow -\infty} \| \frac{1}{t \lambda(1/t)} Q(\frac{x }{t \lambda(1/t)}) \frac{|x|^{2}}{4t} \|_{L^{2}} = 0,
\end{equation}
as well as
\begin{equation}\label{14.27}
\lim_{t \searrow -\infty} \| \frac{1}{t^{d/2} \lambda(1/t)^{d/2}} Q(\frac{x }{t \lambda(1/t)}) (e^{i |x|^{2}/4t} - 1)\|_{L^{2}} = 0,
\end{equation}
Therefore, $v$ is a solution that blows up backward in time at $\inf(I) = -\infty$ and $\| v \|_{L^{2}} = \| Q \|_{L^{2}}$. Therefore, by Theorem $\ref{t8.1}$, $v$ is a soliton, and $u$ is a pseudoconformal transformation of the soliton.

\section{Acknowledgements}
During the writing of this paper, the author was partially supported by NSF grant DMS-2153750.

\bibliography{biblio}

\begin{thebibliography}{Dod21b}

\bibitem[BHS12]{2012standing}
Jaeyoung Byeon, Hyungjin Huh, and Jinmyoung Seok.
\newblock Standing waves of nonlinear {S}chr{\"o}dinger equations with the
  gauge field.
\newblock {\em Journal of Functional Analysis}, 263(6):1575--1608, 2012.

\bibitem[Dod15]{dodson2015global}
Benjamin Dodson.
\newblock Global well-posedness and scattering for the mass critical nonlinear
  {S}chr{\"o}dinger equation with mass below the mass of the ground state.
\newblock {\em Advances in mathematics}, 285:1589--1618, 2015.

\bibitem[Dod19]{dodson2019defocusing}
Benjamin Dodson.
\newblock {\em Defocusing {N}onlinear {S}chr{\"o}dinger {E}quations}, volume
  217.
\newblock Cambridge University Press, 2019.

\bibitem[Dod21a]{dodson2021determination2}
Benjamin Dodson.
\newblock A determination of the blowup solutions to the focusing nls with mass
  equal to the mass of the soliton.
\newblock {\em arXiv preprint arXiv:2106.02723}, 2021.

\bibitem[Dod21b]{dodson2021determination}
Benjamin Dodson.
\newblock A determination of the blowup solutions to the focusing, quintic
  {NLS} with mass equal to the mass of the soliton.
\newblock {\em arXiv preprint arXiv:2104.11690}, 2021.

\bibitem[Dod21c]{dodson20212}
Benjamin Dodson.
\newblock The ${L}^{2}$ sequential convergence of a solution to the
  mass-critical {NLS} above the ground state.
\newblock {\em arXiv preprint arXiv:2101.09172}, 2021.

\bibitem[Dod22]{dodson2022l2}
Benjamin Dodson.
\newblock The ${L}^{2}$ sequential convergence of a solution to the
  mass-critical {NLS} above the ground state.
\newblock {\em Nonlinear Analysis}, 215:112612, 2022.

\bibitem[Dod23]{dodson2023sequential}
Benjamin Dodson.
\newblock Sequential convergence of a solution to the
  chern--simons--schrodinger equation.
\newblock {\em arXiv preprint arXiv:2309.10925}, 2023.

\bibitem[Fan21]{fan20182}
Chenjie Fan.
\newblock The ${L}^{2}$ {W}eak {S}equential {C}onvergence of {R}adial
  {F}ocusing {M}ass {C}ritical {NLS} {S}olutions with {M}ass {A}bove the
  {G}round {S}tate.
\newblock {\em Int. Math. Res. Not. IMRN}, (7):4864--4906, 2021.

\bibitem[GV92]{ginibre1992smoothing}
Jean Ginibre and Giorgio Velo.
\newblock Smoothing properties and retarded estimates for some dispersive
  evolution equations.
\newblock {\em Communications in mathematical physics}, 144(1):163--188, 1992.

\bibitem[KKO22]{kim2022soliton}
Kihyun Kim, Soonsik Kwon, and Sung-Jin Oh.
\newblock Soliton resolution for equivariant self-dual
  {C}hern-{S}imons-{S}chr{\"o}dinger equation in weighted {S}obolev class.
\newblock {\em arXiv preprint arXiv:2202.07314}, 2022.

\bibitem[KTV09]{killip2009cubic}
Rowan Killip, Terence Tao, and Monica Vișan.
\newblock The cubic nonlinear schr{\"o}dinger equation in two dimensions with
  radial data.
\newblock {\em Journal of the European Mathematical Society}, 11(6):1203--1258,
  2009.

\bibitem[LL22]{li2022threshold}
Zexing Li and Baoping Liu.
\newblock On threshold solutions of the equivariant
  {C}hern--{S}imons--{S}chr{\"o}dinger equation.
\newblock {\em Annales de l'Institut Henri Poincar{\'e} C}, 39(2):371--417,
  2022.

\bibitem[LS16]{liu2016global}
Baoping Liu and Paul Smith.
\newblock Global wellposedness of the equivariant
  {C}hern--{S}imons--{S}chr{\"o}dinger equation.
\newblock {\em Revista Matem{\'a}tica Iberoamericana}, 32(3):751--794, 2016.

\bibitem[LST14]{liu2014local}
Baoping Liu, Paul Smith, and Daniel Tataru.
\newblock Local wellposedness of {C}hern--{S}imons--{S}chr{\"o}dinger.
\newblock {\em International Mathematics Research Notices},
  2014(23):6341--6398, 2014.

\bibitem[MR05]{merle2005blow}
Frank Merle and Pierre Raphael.
\newblock The blow-up dynamic and upper bound on the blow-up rate for critical
  nonlinear {S}chr{\"o}dinger equation.
\newblock {\em Annals of mathematics}, pages 157--222, 2005.

\bibitem[Tao00]{tao2000spherically}
Terence Tao.
\newblock Spherically averaged endpoint {S}trichartz estimates for the two
  dimensional {S}chr{\"o}dinger equation.
\newblock {\em Communications in Partial Differential Equations},
  25(7-8):1471--1485, 2000.

\bibitem[Yaj87]{yajima1987existence}
Kenji Yajima.
\newblock Existence of solutions for {S}chr{\"o}dinger evolution equations.
\newblock {\em Communications in Mathematical Physics}, 110(3):415--426, 1987.

\end{thebibliography}
\bibliographystyle{alpha}
\end{document}